\documentclass[11pt]{article}
\usepackage{amsfonts,amsbsy,amssymb,amsmath,graphicx}
\usepackage{geometry}
\usepackage{color}
\usepackage[colorlinks=true,linkcolor=red,citecolor=blue]{hyperref}
\usepackage{bbm}
\usepackage{natbib}
\usepackage{verbatim}
\newtheorem{propo}{Proposition}[section]
\newenvironment{proof} {\noindent {\em \textbf{Proof}} } { \hfill \fbox{~} \\ }

\def\1{1\kern-.20em {\rm l}}
\newcommand{\ER}{\ensuremath{\mathbb{R}}}

\newtheorem{theorem}{Theorem}[section]

\newtheorem{corollary}[theorem]{Corollary}

\newtheorem{lemma}[theorem]{Lemma}

\newtheorem{remark}[theorem]{Remark}

\newcommand{\argmin}{\mathop{\mathrm{arg\,min}}}

\begin{document}

\title {Nonparametric Multivariate $L_1$-median Regression Estimation with Functional Covariates}

\author{{\sc Mohamed Chaouch}$^{1,}\footnote{corresponding author}$ \quad and \quad {\sc Na\^amane La\"ib}$^2$ \\
$^1$ Centre for the Mathematics of Human Behaviour (CMoHB)\\
Department of Mathematics and
Statistics, University of Reading, UK\\
$^2$Laboratoire de Statistique Th\'eorique et Appliqu\'ee, Universit\'e Paris 6, France\\
email : $m.chaouch$@reading.ac.uk, $naamane.laib$@upmc.fr}

\maketitle

\begin{abstract}

In this paper, a nonparametric estimator is proposed for
estimating the $L_1$-median for multivariate conditional
distribution when the covariates take values in an {\it infinite}
dimensional space. The multivariate case is more
appropriate to predict the components of a vector of random
variables simultaneously rather than predicting each of them
separately. While estimating the conditional $L_1$-median function
using the well-known Nadarya-Waston estimator, we establish the
strong consistency of this estimator as well as the asymptotic
normality. We also present some simulations and provide how to
built conditional confidence ellipsoids for the multivariate
$L_1$-median regression in practice. Some numerical study in chemiometrical real data are carried out to compare the multivariate $L_1$-median regression with the vector of marginal median regression when the covariate $X$ is a curve as well as $X$ is a random vector. 


\end{abstract}

\noindent \textbf{Keywords}: almost sure convergence, confidence ellipsoid, functional data, kernel estimation, small balls probability, multivariate conditional $L_1$-median, multivariate conditional distribution.

\section{Introduction }

 In statistics, researchers are often interested in how a
variable response $Y$ may be concomitant with an explanatory
variable $X$. Studying the relationship between $Y$ given a new
value of the explanatory variable $X$ is an important task in
non-parametric statistics. For instance, regression function
provides the mean value that takes $Y$ given $X=x$. Some other characteristics of the conditional distribution, such as conditional median, conditional quantiles, conditional
mode, maybe quite interesting in practice. Furthermore, it is widely acknowledged that
quantiles are more robust to outliers than regression function.

Conditional quantiles are widely studied when the explanatory
variable $X$ lies within a {\it finite} dimensional space. There
are many references on this topic (see \cite{gannoun2003}).

During the last decade, thanks to progress of computing tools,
there is an increasing number of examples coming from different
fields of applied sciences for which the {\it data are curves}.
For instance, some random variables can be observed at several
different times. This kind of variables,  known as {\it functional
variables} (of time for instance)
  in the literature,
 allows us to consider the data  as curves. The books by \cite{bosq2000} and \cite{ramsay2005}) propose an interesting description
 of the available procedures dealing with functional observations whereas \cite{ferraty2006}
 present a completely non-parametric point of view.
 These functional approaches mainly rely on generalizing multivariate statistical procedures in functional
 spaces and have been proved to be useful in various  areas  such as chemiomertrics (\cite{hastie} and \cite{quintela2011}), economy (\cite{kneip}),
 climatology (\cite{besse}), biology (\cite{Kirkpatrick}), Geoscience (\cite{quintela2011}) or hydrology (\cite{chebana}).
 These functional approaches are generally more appropriate than longitudinal data models
 or time series analysis when there are, for each curve, many measurement points (\cite{rice}).
 
In the {\it univariate} case (i.e. $Y \in \ER$ and $X$ is a
functional covariable), among the lot of papers
dealing with the nonparametric estimation of conditional
quantiles, one may cite papers by \cite{cardot2005} which introduced univariate quantile regression with functional covariate and \cite{ferraty2005} estimates conditional quantile by inverting the conditional
cumulative distribution function.
 \cite{ezzahrioui} establish the almost
 complete convergence and the asymptotic normality in the setting of independent and identically distributed (i.i.d.) data as well as
 under  $\alpha$-mixing condition. \cite{daboniang} stated the convergence in $L^p$-norm.  In the same framework,
 \cite{laksaci} estimated the conditional quantile nonparametrically, by adapting the $L^1$-norm method.
Recently \cite{quintela20112}
have used the same approach proposed by \cite{ferraty2005} to predict future
stratospheric ozone concentrations and to estimate return levels
of extreme values of tropospheric ozone.

 Over the past decades, researchers have shown increasing interest
in studying {\it multivariate} location parameters such as
multivariate quantiles in order to find suitable analogs of
univariate quantiles that used to construct descriptive statistics
and robust estimations of location. In contrast to the univariate
case, the order of observations $Y_i$ laying in  $\ER^{d}$  (with
$d\geq 2$) is not total. Consequently, several quantiles-type
multivariate definitions have been formulated. The pioneer paper
of \cite{haldane} considered a
multivariate extension of the median defined as an $M$-estimator
(also called spatial or $L_1$-median). The reader is referred to \cite{serfling} for historical reviews and
comparisons. \cite{chaudhuri1996} and \cite{koltchinskii} defined the
geometric quantile as an extension of multivariate quantiles based
on norm minimization and on the geometry of multivariate data
clouds.

 In contrast, relative little attention has been paid to
the multivariate conditional quantiles ($Y\in \ER^d$ and $X\in \ER^s$) and their large sample
properties. \cite{cadre} defined the conditional $L_1$-median and provided its uniform consistency on a compact subsets of $\ER^s$. Recently, \cite{degooijer} have introduced a multivariate conditional quantile notion, which
extends the definition of unconditional quantiles by \cite{abdous}, to predict tails from bivariate time series.
\cite{cheng} have generalized the notion of geometric quantiles, defined by \cite{chaudhuri1996}, to the conditional setting. They  have
established a Bahadur-type linear representation of the
$u$-th geometric conditional estimator as well as the
asymptotic normality  in the i.i.d. case.

The purpose of this paper is to add some new results to the
non-parametric estimation of the conditional $L_1$-median when $Y$
is a random vector with values in $\ER^d$ while the covariable $X$
take its values in some {\it infinite} dimensional space
$\mathcal{F}$.
 As far as we know, this problem has not been studied in
 literature before and the results obtained here are believed to be novel. Moreover,  our motivation for studying this type of robust
 estimator is due to its interest in some practical applications. Note also that, it would be better to predict all
components  of a vector of random variables simultaneously in order to
take into account the correlation between them rather than predicting
each of component  separately.  For instance, in EDF (French
electricity company) the estimation of the minimum and the maximum
of the electricity power demand represents an important research
issue for both economic and security reasons.  Because an underestimation
of the maximum consumed quantity of electricity (especially in winter)  may require importation of
electricity from  other  European countries with high prices, while
an over estimation of this maximum quantitiy may  induce a negative effect
on the electricity distribution network. The estimation of the
minimum power demand is also an important task for the same
reasons. Notice that the minimum and the maximum of the
electricity power demand are strongly correlated. Thus, it is more appropriate to predict these variables simultaneously rather
than predicting each of them separately. On the other hand,
weather variables, like temperature curves,  can play a key role to
explain the minimum and the maximum of power demand. Due to its
robust properties, the conditional $L_1$-median may  be used
to solve this prediction problem using a temperature curve as
covariate.

The  paper is organized as follows. Section 2 outlines notations
and the form of the new estimator. Section 3 presents the main
results concerning the asymptotic behavior of the estimator,
including consistency, asymptotic normality and evaluation of the
bias term. An estimation of the conditional confidence region is
then deduced. Section 4 is devoted to a simulation study giving an
example of the estimated confidence region. An application to
chemiometrical real data is proposed in Section 5, where we
compare three approaches: $L_{1}$-median regression, the vector of marginal conditional median and non-functional multivariate median
to predict a random vector. The proofs of the results in Section
3 are relegated to the Appendix.


\section{Notations and definitions}


Let us consider a random pair $(X, Y)$ where $X$ and $Y$ are two
random variables defined on the same probability space $(\Omega,
\mathcal{A}, \mathbb{P})$. We suppose that $Y$ is $\mathbb{R}^d$-valued and $X$ is a \textit{functional random variable
(f.r.v.)} takes its values in some infinite dimensional vector space $(
\mathcal{F}, d(\cdot,\cdot))$ equipped with a semi-metric
$d(\cdot,\cdot)$. Let $x$ be a fixed point in $\mathcal{F}$ and
$F(.|x)$ be the conditional cumulative distribution function
(cond. c.d.f) of $Y$ given $X=x$. The conditional $L_1$-median,
$\mu: {\cal F} \longrightarrow \mathbb{R}^d$,  of $Y$ given $X=x$,
is defined as the miminizer over $u$ of
\begin{eqnarray} \label{qdef}
& & \arg\min_{u \in \ER^d} \mathbb{E} [(\| Y - u\| - \|Y\|) \;| \;
X =x]
  = \arg\min_{u \in \ER^d} \displaystyle\int (\|y - u\| - \|y\|) \; dF(y \;| \;  x).
 \end{eqnarray}

\noindent  The general definition (\ref{qdef}) does not assume the
existence of the first order moment of $\|Y\|$.  However, when $Y$
has a finite expectation, $\mu(x)$ becomes a minimizer over $u$ of
$ \mathbb{E} [\|Y - u\| \;| \;  X =x].$  Notice that the
existence and the uniqueness of $\mu(x)$ is guaranteed, for $d
\geq 2,$ provided that the conditional distribution function
$F(\cdot | x)$ is not supported on a single straight line (see
theorem 2.17  of  \cite{kemperman}. Hence, uniqueness holds
whenever $Y$ has an absolutely continuous conditional distribution
on $\ER^d$ with $d \geq 2.$

\noindent Without loss of generality, we suppose in the sequel,
that $\mathbb{E}\|Y\|<\infty$.  Therefore  for any fixed $x\in
\mathcal{F}$, the conditional $L_1$-median $\mu(x)$  may  be
viewed as a minimizer of the function $G^x : \ER^d \longmapsto
\ER$ defined, for all $u\in\ER^d$, by
\begin{equation}\label{G}
G^x(u) :=\mathbb{E} [\|Y - u\| \;| \;  X =x],
\end{equation}

\noindent which is assumed to be differentiable and uniformly bounded with respect to $u$.

 We introduce now some further definitions and
 notations. Denote by   $A^t$ the transpose of  the matrix $A$, and let $\|A\| = \sqrt{tr(A^{t}\, A)} $ be the norm trace.
 Notice that for any $y\in \mathbb{R}^d$,
the function $y \longmapsto \|y\|$ is differentiable everywhere
except at $z=0_{\mathbb{R}^d}$, one may then  define (by
continuity extension) its derivative as $\mathcal{U}(y)=y/\|y\|$
when  $y\neq 0$ and $\mathcal{U}(y) = 0$ whenever
 $y=0$. For any  $y\neq u$,
 define
  $$\mathcal{M}(y,u) = (1/\|y-u\|)({\bf I}_d -
 \mathcal{U}(y-u)\mathcal{U}^t(y-u)),$$
 where ${\bf I}_d$ is the  $d\times d$ identity matrix. We denote by $\nabla_u G^x(u)$ the gradian of the function
$G^x(u)$ and by $H^x(u)$ its Hessian functional matrix (with
respect to $u$). According to \cite{koltchinskii}, it is easy to
see that
\begin{eqnarray}\label{gradian}
\nabla_u G^x(u) = - \mathbb{E}\left[\mathcal{U}(Y-u)\; | \; X=x
\right] \quad \mbox{and}
\end{eqnarray}
\begin{eqnarray}\label{H}
H^x(u) = \mathbb{E}\left[\mathcal{M}(Y,u) \;|\; X=x\right].
\end{eqnarray}
Notice that $H^x(u)$ is bounded whenever
$\mathbb{E}\left[\|Y-u\|^{-1}\mid X=x\right] < \infty.$ According to
(\ref{qdef}) and (\ref{gradian}), the conditional $L_1$-median
may be then implicitly defined as a zero with respect to $u$ of
the following equation:
\begin{eqnarray}\label{zero}
\nabla_u G^x(u) = 0.
\end{eqnarray}

To build our estimator, let $(X_i, Y_i)_{i = 1, \dots, n}$ be the
statistical sample of pairs which are independent and identically
distributed as $(X, Y)$. Let us denote by
$$w_{n,i}(x) = \frac{\Delta_i(x)}{\sum_{i=1}^n \Delta_i(x)},$$
the so-called Nadaraya-Watson weights, where $\Delta_i(x) =
 K\left(d(x,X_i)/h\right)$, with $K$ a kernel function,
  $h :=h_n$ is a sequence of positive real numbers which decreases to zero as $n$ tends to infinity.



\noindent A  kernel estimator of the function $G^x(u)$ is given by
\begin{equation}
G_{n}^x(u) = \sum_{i=1}^n w_{n,i}(x) \,\| Y_{i} - u \| = \frac{\sum_{i=1}^n \|Y_i - u\|\;
\Delta_i(x)}{\sum_{i=1}^n \Delta_i(x)}
:=\frac{G^x_{n,2}(u)}{G^x_{n,1}},
\end{equation}

\noindent when the denominator is not equal to 0, where
\begin{equation}
G^x_{n,j}((j-1)u) = \frac{1}{n\mathbb{E}(\Delta_1(x))}\sum_{i=1}^n
\|Y_i - u\|^{j-1} \Delta_i(x), \quad \mbox{for} \quad j=1,2 \quad
\mbox{with}\quad  G^x_{n, 1}(0):=  G^x_{n,1}.
\end{equation}
A kernel estimate  of $\nabla_u G^x(u)$  may be  defined by
\begin{eqnarray}\label{estgratian} \nabla_u G^x_n(u) := - \sum_{i=1}^n
w_{n,i}(x)\;\mathcal{U}(Y_i-u), \quad  u\in\ER^d.
\end{eqnarray}

According to the statement (\ref{G}), the estimator of the conditional $L_1$-median,
$\mu_n(x)$,  may be viewed as  a minimizer over $u$ of the
function $G^x_n(u)$, that is
\begin{equation}
\mu_n(x) = \arg\min_{u \in \ER^d} G^x_n(u),
\end{equation}
or as a zero with respect to $u$ of the equation $ \nabla_u G^x_n(u) = 0. $

 \noindent Similar to the Fact $2.1.1$ in \cite{chaudhuri1996}
and Remark $2.3$ in \cite{cheng}, the existence of
the estimator $\mu_n(x)$ is guaranteed by the fact that the
function $u\longmapsto\sum_{i=1}^n w_{n,i}(x) \|Y_i - u\|$
explodes to infinity as $|| u|| \rightarrow \infty$. On the other
hand, since this function is continuous with respect to $u$, then
$\mu_n(x)$ must be a minimizer over $u$ of $\sum_{i=1}^n
w_{n,i}(x) \|Y_i - u\|$. Next comes the question of uniqueness,
since $\ER^d$  is  equipped with the Euclidean norm that is a
strictly convex Banach space for $d\geq2$, it follows from Theorem
$2.17$ of \cite{kemperman} that unless all the data points $Y_1,
\dots, Y_n$ fall on a straight line in $\ER^d$, $\sum_{i=1}^n
w_{n,i}(x) \|Y_i - u\|$ must be a strictly convex function of $u$.
 This guarantees the uniqueness of the minimizer
 $\mu_n(x)$ in $\ER^d$, for any $d \geq2$.


\section{Main Results}
\subsection{Further notations and hypotheses}

Let $x$ be a given point in ${\cal F}$  and
$\mathcal{V}_x$  a  neighbourhood of $x$. Denote by
$\mathcal{B}(x, h)$ the ball of center $x$ and radius $h$, namely
$\mathcal{B}(x,h) = \{ x' \in \mathcal{F} : d(x,x') \leq h \}$.
For $(\ell, u)  \in \mathbb{R}\times \mathbb{R}^d $, denote by
$G_{\ell}^{x^\prime} (u)= \mathbb{E}\left[ \| Y- u\|^{\ell}\mid
X=x^\prime\right]$, for $x^{\prime} \in \mathcal{F}$. Our
hypotheses are gathered here for easy reference.

\begin{itemize}
\item[(H1)] $K$ is a nonnegative bounded kernel of class
  $\mathcal{C}^1$ over its support $[0,1]$ such that $K(1) >
  0.$ The derivative $K'$ exists on $[0,1]$ and satisfy the
  condition $K'(t) < 0,$  for all $t \in [0,1]$ and $|\int_0^1 (K^j)'(t) dt | < \infty$ for $j=1,2.$

\item[(H2)] For $x \in \mathcal{F}$, there exists a deterministic nonnegative bounded function $g$ and a nonnegative real function $\phi$ tending to zero, as its argument tends to 0, such that
\begin{itemize}
\item[$(i)$] $F_x(h) := \mathbb{P}(X \in \mathcal{B}(x,h)) = \phi(r)\cdot g(x) + o(\phi(h))$ as $h \rightarrow 0.$
\item[$(ii)$]There exists a nondecreasing bounded function $\tau_0$ such that, uniformly in $s \in [0,1]$,

$\displaystyle\frac{\phi(hs)}{\phi(h)} = \tau_0(s) + o(1),$ as $h \downarrow 0$ and, for $j \geq 1$, $\int_0^1 (K^j(t))' \tau_0(t) dt < \infty.$
\end{itemize}
\item[(H3)]
\begin{itemize}
\item[$(i)$] For $x \in \mathcal{F}$, $|G^x(u) - G^{x'}(u)| \leq
c_1 d^{\beta}(x,x')$ uniformly in $u$, for some $\beta>0$ and a
constant $c_1>0$, whenever $x' \in \mathcal{V}_{x}$,
 \item[$(ii)$] For $x^\prime \in \mathcal{F}$, the Hessian matrix
$H^{x^\prime}(u)$ is continuous in $\mathcal{V}_x$:

$\sup_{x'\in B(x,h)}\|H^{x}(u) -
H^{x'}(u)\|=o(1)$.
\item[$(iii)$] For some integer $m\geq 2$,
$G_{-m}^{x}(\mu(x)) <\infty$ and $G_{-m}^{x^\prime}(\mu(x))$ is continuous in $\mathcal{V}_x$.


\item[$(iv)$]  For some integer $m\geq 1$ and any $(k, j)$, $1\leq
k\leq d$, $1\leq j\leq d$, $\mathbb{E}\left[
\mathcal{M}_{k,j}^{m}(Y,\mu) \mid X)\right] < \infty$ and
$$\sup_{\{x^\prime : d(x,x^\prime) \leq h\}} \left|\mathbb{E}\left( \mathcal{M}_{k,j}^{m}(Y,\mu) \mid x^\prime)\right) -
\mathbb{E}\left( \mathcal{M}_{k,j}^{m}(Y,\mu) \mid x)\right)
\right| = o(1).$$
\end{itemize}
\item[(H4)] $(i)$ For each $x^\prime\in{\cal F}$,  $\sup_u G_{m}^{x^\prime}(u)
<\infty$ and $G_m^{x^\prime}(u)$ is continuous in $\mathcal{V}_{x}$ uniformly in $u$:
$$\displaystyle \sup_{u\in\mathbb{R}^{d}}\sup_{\{x^\prime: d(x,x') \leq h\}} |G^{x'}_m(u) - G^x_m(u)| = o(1).$$

 $(ii)$ For some $\delta >0$ and $\ell \in \ER^d$, the real
function $W_{i+j\delta}^{x^\prime}(\mu) := \mathbb{E}\left[|\ell^t
\mathcal{U}(Y-\mu)|^{i+j\delta} \;|\; X =x^\prime\right]$  ($i=1,
2$ and $j=0, 1$) is continuous in $\mathcal{V}_x.$

 \end{itemize}

\begin{itemize}
\item[ (H5)] For  any $i \geq 1$,
$\mathbb{E}\left[\mathcal{U}(Y-\mu)\;|\; d(x,X) =v\right] =:
\psi(v))$,
 where $v\in \mathbb{R}$ and $\psi : \ER \rightarrow \ER^d$ is a  differentiable function
 such that $\nabla\psi(0) \neq 0$.

\end{itemize}

\begin{remark}
\noindent  Notice that, since $d(\cdot, \cdot)$ is a semi-metric,
we have
$\psi(0)=\mathbb{E}\left[\mathcal{U}(Y-\mu)\;|\;X=x\right]$. As a
consequence, it follows from the definition of $\mu$ that
 $\psi(0)=0$.

\end{remark}
\noindent {\bf Comments on the Hypotheses}
\medskip

The above conditions are fairly mild. Condition (H1) is standard
in the context of functional non-parametric estimation. Contrarily to the real and vectorial cases
(for which we generally suppose the strict positivity of the explanatory variable's density,
 the concentration hypothesis (H2)-(i) acts directly on the distribution of the functional random variable rather than
on its density function. The idea of writing the small ball probability $F_{x}(h)$
as a product of two independent functions $g(x)$ and $\phi(h)$ was adopted by \cite{masry}
who reformulated the \cite{gasser} one. This assumption has been used by many authors where $g(x)$ is interpreted as a
 probability density, while $\phi(h)$ may be interpreted as a volume parameter. In the case of finite-dimensional space,
 that is $\mathcal{F}=\mathbb{R}^d$, it can be seen that $F_{x}(h) = C(d) h^d g(x) + o(h^d)$, where $C(d)$ is
 the volume of the unit ball in $\mathbb{R}^d$. Furthermore, in {\it infinite}  dimensions, there exist many examples
 fulfilling the decomposition mentioned in assumption
 (H2)-(i) (see \cite{ferraty2007} and \cite{ezzahrioui} for more details).
  The function $\tau_{0}(\cdot)$, introduced in assumption (H2)-(ii), plays a determinant role in asymptotic properties,
  in particular when we give the order of the conditional bias and the asymptotic variance term.

Conditions (H3) and (H4) are mild smoothness assumptions on the functionals $G^{(\cdot)}(u)$ and $H^{(\cdot)}(u)$ and continuity assumptions on certain second-order moments. A similar assumption to (H3)-(iii) has been supposed in \cite{cheng} (see condition 6 in their paper). Condition (H5) is used to evaluate the bias term.



\subsection{Almost sure consistency}
The following result states the almost surely (a.s.)
convergence (with rate)  of the functional estimator $G^x_n(u)$.
This result plays an instumental role to prove the almost sure consistency
of $\mu_n(x)$ for a fixed $x \in{\cal F}$.

\begin{propo}\label{resG}
Assumes that conditions (H1)-(H2), (H3)(i) and (H4)(i)  hold
true and
\begin{equation}\label{CondThm1}
(i) \quad \frac{\log n}{n\phi(h)}\to 0 \quad \mbox{and} \quad (ii)
\ \frac{n\phi(h) h^{2\beta}}{\log n}\to 0 \quad \mbox{as} \quad
n\to \infty, \;\mbox{where} \; \beta\; \mbox{is is given in $(H3)$},
\end{equation}
%
\begin{equation}\label{Cond3.Thm1}
\overline{\lim}_{||u||\to \infty}||u||G^x(u)<\infty.
\end{equation}
 Then, we have
$$\sup_{u\in\ER^d} \left|G^x_n(u) - G^x(u)\right| = O_{a.s}(h^\beta)+O_{a.s}\left(\sqrt{\frac{\log n}{n\phi(h)}} \right). $$
\end{propo}

Notice that the condition (\ref{Cond3.Thm1}) is standard when we
deal with the uniform consistency of the density function on the
 whole space (see, for instance, Corollary 2.2 of
 \cite{bosq1996}).

 \noindent Here then, we give our first result of the
conditional $L_1$-median estimator $\mu_n(x)$.

\begin{theorem}\label{thm2}
Assume  (H1)-(H2), (H3)(i) and  (H4)(i)  and condition
(\ref{CondThm1}) hold true. Then, we have
\begin{equation}\label{Con.median}
\lim_{n\rightarrow\infty}\mu_n(x) = \mu(x) \quad \mbox{a.s.}
\end{equation}
\end{theorem}

\subsection{Asymptotic normality}
To state the asymptotic normality of our estimator, some notations are required.
 Let us first denote by
 $$\widetilde{G}_{n}^x(u) =  \frac{\sum_{i=1}^n \| Y_{i} - u \| \,
 \Delta_{i}(x)}{n\;\mathbb{E}(\Delta_{1}(x))} \quad \mbox{and}
 \quad
 \nabla_{u}\widetilde{G}_{n}^x (u) = - \frac{\sum_{i=1}^n \mathcal{U}(Y_{i} - u) \Delta_{i}(x)}{n\mathbb{E}(\Delta_{1}(x))}. $$
Set   $\mu(x) =: \mu= (\mu_1, \dots,\mu_d)^t$ and $\mu_n(x) =:
\mu_n=(\mu_{n,1}, \dots, \mu_{n,d})^t$. We have by the definition
 of $\mu_n$ that
\begin{eqnarray}
\label{equivo}
\nabla_u G_n^x(\mu_n) = -\; \frac{\sum_{i=1}^n\mathcal{U}(Y_i-\mu_n) \Delta_i(x)}{\sum_{i=1}^n \Delta_i(x)} =0.
\end{eqnarray}
Obviously the equation (\ref{equivo}) is satisfied when the numerator is null. Then, we can say also that
\begin{eqnarray}
\label{equivo2}
\nabla_u \widetilde G_n^x(\mu_n) = - \frac{\sum_{i=1}^n\mathcal{U}(Y_i-\mu_n) \Delta_i(x)}{n\; \mathbb{E}(\Delta_1(x))} =0.
\end{eqnarray}
\noindent Thereafter, one may write
\begin{eqnarray}\label{decomposition}
\nabla_u \widetilde G_n^x(\mu_n) - \nabla_u \widetilde G_n^x(\mu)= - \nabla_u \widetilde G_n^x(\mu).
\end{eqnarray}

\noindent For each $j\in\{1,\dots,d\}$, Taylor's expansion applied
to the real-valued function  $\displaystyle\frac{\partial
\widetilde G_n^x}{\partial u_j}$ implies the existence of $\xi_n(j)
= (\xi_{n,1}(j), \dots,\xi_{n,d}(j))^t$ such that

$$\left\{\begin{array}{l}
\displaystyle\frac{\partial \widetilde G_n^x}{\partial u_j}(\mu_n) -
\frac{\partial \widetilde G_n^x}{\partial u_j}(\mu) =
 \sum_{k=1}^d \frac{\partial^2 \widetilde G_n^x}{\partial u_j\partial u_k} (\xi_n(j)) (\mu_{n,k}-\mu_k),\\
|\xi_{n,k}(j) - \mu_k| \leq |\mu_{n,k}(j) - \mu_k|.
\end{array}\right. $$

\medskip

\noindent Define the $d\times d$ matrix $\widetilde H_n^x(\xi_n(j))=(\widetilde H_{n,k,j}^x(\xi_n(j)))_{1\leq k,j\leq d}$ by setting

$$\widetilde H_{n,k,j}^x(\xi_n(j)) = \displaystyle\frac{\partial^2 \widetilde G_n^x}{\partial u_j\partial u_k} (\xi_n(j)),$$

\noindent where, for all $u\in \ER^d$ and $x\in\mathcal{F}$,
\begin{eqnarray*}
\widetilde H_{n,k,j}^x(u) &=& \sum_{i=1}^n \frac{1}{\|Y_i-u\|} \left[\delta_{k,j} - \frac{(Y_i^j-u^j)(Y_i^k-u^k)}{\|Y_i-u\|^2} \right] \times \frac{\Delta_i(x)}{n\;\mathbb{E}(\Delta_1(x))} \\
&=& \displaystyle\frac{\sum_{i=1}^n \mathcal{M}_{k,j}(Y_i,u) \Delta_i(x)}{n\;\mathbb{E}(\Delta_1(x))},
\end{eqnarray*}
with  $\delta_{k,j} = 1$ if $k=j$ and zero otherwise and
$\mathcal{M}_{k,j}(Y_i,u) =
[\delta_{k,j}-\frac{(Y_i^j-u^j)(Y_i^k-u^k)}{\|Y_i-u\|^2}]/\|Y_i-u\|$
is the $(k,j)$-th element of the matrix $\mathcal{M}(Y_i,u)$.
Equation (\ref{decomposition}) can be then rewritten as
\begin{eqnarray}\label{decomposition1}
\widetilde H_n^x(\xi_n(j))\;(\mu_n - \mu) = - \nabla_u \widetilde G_n^x(\mu).
\end{eqnarray}

\noindent Equation (\ref{decomposition1})  plays a key role to
give the conditional bias and the asymptotic distribution of the conditional
$L_1$-median estimator $\mu_n$.


\begin{propo}\label{variance}
Under assumptions (H1)-(H3) and  (H4)(i) and condition
(\ref{CondThm1})(i), we have
\begin{eqnarray*}
 \| \widetilde H_n^x(\xi_n(j)) - H^x(\mu)\| = o_{\mathbb{P}}(1), \quad\quad{as}\; n\rightarrow\infty.
\end{eqnarray*}
\end{propo}

\vskip 2mm

 \noindent  Using  Remark 4 and Lemma $5.3$ of
\cite{chaudhuri1992}, we know that both the matrix $H^x(\mu)$ itself
and its inverse matrix exist whenever  $d\geq2$. It follows  from
this result combined with (\ref{decomposition1}) that, for n large enough, $\mu_n-\mu=
-[H^x(\mu)]^{-1}\nabla_u \widetilde G_n^x(\mu)+ o_{\mathbb{P}}(1)$. One may then write, for large $n$ that
\begin{eqnarray}\label{deco}
\sqrt{n\phi(h)} \left(\mu_n - \mu \right) =\sqrt{n\phi(h)} \left[H^x(\mu)\right]^{-1}  \left[- \left(\nabla_u \widetilde G_n^x(\mu) -
 \mathbb{E}\left[\nabla_u \widetilde G_n^x(\mu)\right] \right) -
 \widetilde{\mathcal{B}}_n(x)\right] + o_{\mathbb{P}}(1),
\end{eqnarray}
where  $\widetilde{\mathcal{B}}_n(x)=\mathbb{E}\left[\nabla_u \widetilde
G_n^x(\mu)\right].$

%

\medskip

\noindent The following  proposition gives  the order of the conditional bias  term
$\mathcal{B}_{n}(x) = - \left[ H^x(\mu) \right]^{-1}\widetilde{\mathcal{B}}_n(x)$.

\begin{propo}\label{B}
Under assumptions $(H1)$, $(H2)$ and  $(H5)$, and the fact that\\
$g(x) > 0$ and $|\int_0^1 (sK(s))' \tau_0(s)ds| < \infty$, we
have:

$$ \mathcal{B}_{n}(x) =
\frac{h \left[H^x(\mu)\right]^{-1} \nabla\psi(0)}{M_1}\left[
\int_0^1 (sK(s))' \tau_0(s) ds - K(1) + o_{a.s.}(1) \right], $$
where for $j=1,2$,  $M_j = K^j(1) - \int_0^1 (K^j)' \tau_0 (z)dz. $

\end{propo}

\noindent The Theorem below gives the asymptotic normality of our estimator.

\begin{theorem}\label{CLT}
Suppose assumptions (H1)-(H5) and condition (\ref{CondThm1})(i) hold. 

If $(n\phi(h))^{\delta/2} \rightarrow \infty ,$ for some $\delta >0$, then:

\begin{itemize}
\item[$(i)$] $\sqrt{n \phi(h)}\; \left(\mu_n(x) - \mu(x) - \mathcal{B}_{n}(x) \right) \stackrel{\mathcal{D}}{\longrightarrow} \mathcal{N}_d\left( 0, \Gamma^x(\mu)\right),$

where
$$\Gamma^x(\mu) = \frac{M_2}{M_1^2 g(x)} \left[ H^x(\mu)\right]^{-1} \Sigma^x(\mu)  \left[ H^x(\mu)\right]^{-1}$$
and
\begin{eqnarray*}
\Sigma^x(\mu) = \mathbb{E}\left\{ \mathcal{U}(Y-\mu)
\;\mathcal{U}^t(Y-\mu) |\; X=x\right\}.
\end{eqnarray*}
\item[$(ii)$] If in addition we impose the following  stronger conditions
on the bandwidth $h_n$:
 $$\sqrt{n\phi(h)} h\longrightarrow 0 \quad \mbox{ as} \quad
n\rightarrow \infty,$$
one gets
$$\sqrt{n \phi(h)}\; \left(\mu_n(x) - \mu(x) \right) \stackrel{\mathcal{D}}{\longrightarrow} \mathcal{N}_d\left( 0, \Gamma^x(\mu)\right).$$
\end{itemize}
\end{theorem}

\begin{remark}. (i) Notice that the constants $M_1$ and $M_2$ are strictly positive. Indeed making use of
the condition $(H1)$ and the fact that the function
$\tau_0(\cdot)$ is nondecreasing, it suffices to perform a simple
integration by parts. Also, from the point that the conditional
distribution $Y$ given $X=x$ is absolutely continuous, we know
that $\Sigma^x(\mu)$ is definite positive matrix.

\noindent (ii) Whenever $\mathcal{F} = \ER^s$, $s\geq1$, and if
the probability density of the random variable $X$,  say
$g_s(\cdot)$, is of class $\mathcal{C}^1$, then $\phi(h) = V(s)
h^s$, where $V(s)$ is the volume of the  unit ball of $\ER^s$. In
such case, the asymptotic variance expression takes the form

$$\Gamma^x(\mu) = \frac{1}{s g_s(x)} \frac{\int_0^1 K^2(u) u^{s-1} du}{\left(\int_0^1 K(u) u^{s-1} du \right)^2}
\times \left[ H^x(\mu)\right]^{-1} \Sigma^x(\mu)  \left[
H^x(\mu)\right]^{-1}.$$ In such case the central limit theorem has
the form given in the above theorem  with convergence rate
$(nh_n^s)^{1/2}$. Notice that in the finite dimensional case, the
function $\phi(h)$ could decrease to zero as $h\to 0$
exponentially fast and the convergence rate becomes effectively
$(n\phi(h))^{1/2}$. This fact may be used to solve the problem of
the curse of dimensionality (see \cite{masry}, for details). As
an example, consider in an infinite dimensional space setting, the
random process defined by
$$X_t=\theta t+W_t, \quad 0\leq t\leq 1,$$
where $\theta$ is a ${\cal N}(0, 1)$-random variable independent
of the Winer process $W=\{ W_t: 0\leq t\leq 1\}$. It is well-known
(see \cite{lipster}) that the distribution $\nu_X$
of $X$ is absolutely continuous with respect to the Wiener measure
$\nu_X$, which admets  a  Radon-Nikodym density  $f(x)$. In this
case,  hypothesis (H2)(i) is satisfied with
$\phi(h)=\frac{4}{\pi}\exp(-\frac{\pi^2}{8 h^2})$ (see \cite{laib2011} for details). The convergence rate in
Theorem \ref{CLT} being $O(n^{\frac{1-2\alpha}{2}})$ (with $0<\alpha<1/2$) by
taking $h_n :=h=\displaystyle\frac{\pi}{2\sqrt{2}}\frac{1}{\log n^{\alpha}}$.

\end{remark}

Observe now in Theorem \ref{CLT} that the limiting variance
contains the unknown function $g(x)$, therefore the normalization
depends on the function $\phi$ which is not identifiable
explicitly. To make this result operational in practice, we have
to estimate the quantities $\Sigma$, $H$ and $\tau_0.$

\noindent For this purpose, we estimate the conditional variance
matrix $\Sigma^x(\mu)$ of $\nabla_u\widetilde{G}_n^x(\mu)$ by

$$\Sigma_n^x(\mu_n) = \sum_{i=1}^n w_{n,i}(x)\; \mathcal{U}(Y_i-\mu_n)\;\mathcal{U}^t(Y_i-\mu_n),$$

\noindent and the matrix $H^x(\mu)$  by
$$H^x_n(\mu_n) = \sum_{i=1}^n w_{n,i}(x) \mathcal{M}(Y_i,\mu_n).$$

\noindent Making use of the  decomposition of $F_x(u)$ in $
(H2)(i)$, one may estimate $\tau_0(u)$ by

$$\tau_n(u) = \frac{F_{x,n}(uh)}{F_{x,n}(u)}, \quad
\mbox{where} \quad F_{x,n}(u) = \frac{1}{n} \sum_{i=1}^n
\1_{\{d(x,X_i) \leq u \}}.$$
 Subsequently, for a given kernel $K$,
the quantities $M_1$ and $M_2$ are estimated by $M_{1,n}$ and
$M_{2,n}$ respectively replacing $\tau_0$ by $\tau_n$ in their
respective expressions.

\vskip 3mm

 Corollary \ref{TCLP} below, which is a slight
modification of Theorem \ref{CLT}, allows to obtain usefull form of
our results in practice.

\begin{corollary}\label{TCLP}
Assume that conditions of Theorem \ref{CLT} hold true, $K'$ and $(K^2)'$ are integrable functions.
If in addition we suppose that

$nF_x(h) \rightarrow\infty$ \quad and \quad $h^\beta (nF_x(h))^{1/2} \rightarrow 0$, \quad as $n\rightarrow \infty$, \\
\noindent where $\beta$ is specified in the condition $(H3)$, then, for any $x\in\mathcal{F}$ such that $g(x) > 0$, we have
$$\frac{M_{1,n}}{\sqrt{M_{2,n}}}\sqrt{nF_{x,n}(h)} \; \left[\Sigma_n^x(\mu_n)\right]^{-1/2}  H^x_n(\mu_n) \;
\left(\mu_n(x) - \mu(x) \right) \stackrel{\mathcal{D}}{\longrightarrow} \mathcal{N}(0,I_d).$$
\end{corollary}

\subsection{Building Conditional confidence region of $\mu(x)$}

\noindent From Corollary \ref{TCLP}, we can easily see that

$$ \left(\mu_n(x) - \mu(x) \right)^t \left[\Gamma_n^{x}(\mu_n)\right]^{-1} \left(\mu_n(x) - \mu(x) \right) \stackrel{\mathcal{D}}{\longrightarrow} \chi_d^2,$$

\noindent where
$$\left[\Gamma_n^{x}(\mu_n)\right]^{-1} = \frac{M_{1,n}^2 n F_{x,n}(h)}{M_{2,n}}\;H_n^x(\mu_n) \left[\Sigma_n^x(\mu_n) \right]^{-1} H_n^x(\mu_n).$$
Then, the asymptotic $100(1-\alpha)\%$ $(\alpha \in (0,1))$ conditional confidence region for $\mu(x)$ is given by
\begin{eqnarray}\label{ellipse}
 \left(\mu_n(x) - \mu(x) \right)^t \left[\Gamma_n^{x}(\mu_n)\right]^{-1} \left(\mu_n(x) - \mu(x) \right) \leq \chi_d^2(\alpha),
\end{eqnarray}

\noindent where $\chi_d^2(\alpha)$ denotes the $100(1-\alpha)$-th percentile of a chi-squared distribution with $d$ degrees of freedom.


\section{Numerical study}

This section is divided in two parts, in the first one we are interesting in the estimation of conditional confidence ellipsoid of the multivariate $L_{1}$-median regression. The second part is devoted to an application to chemiometrical real data and it consists in predicting a three-dimensional vector.

\subsection{Simulation example}
Let us consider a bi-dimensional vector ${\bf Y} =(Y_1, Y_2) \in \ER^2$ and $X(t)$ is a Brownian motion trajectories defined on $[0,1]$. The eigenfunctions of the covariance operator of $X$ are known to be (see \cite{ash}), for $j=1,2, \dots$
\begin{eqnarray*}
f_j(t) = \sqrt{2} \sin \{(j-0.5)\pi t\}, \; t\in[0,1].
\end{eqnarray*}
Let $(f_1(t))_{t\in [0,1]}$ (resp. $(f_2(t))_{t \in [0,1]}$) be the first (resp. the second) eigenfunction corresponding to the first (resp. second) greater eigenvalue of the covariance operator of $X$. It is well known that $f_1(t)$ and $f_2(t)$ are orthogonal by construction, i.e. $<f_1,f_2> :=\int_{0}^1 f_{1}(t) f_{2}(t) = 0.$


\noindent We modelize then the dependence between ${\bf Y}$ and $X$ by the following model:
\begin{itemize}
\item $Y^1 = \int_0^1 f_1(t) X(t)\; dt + \epsilon$
\item $Y^2 = \int_0^1 f_2(t) X(t)\; dt + \epsilon$
\end{itemize}
where $\epsilon$ is a standard normal random variable.

\begin{figure}[ht!]
\begin{center}
\includegraphics[height=6.5cm,width=17cm]{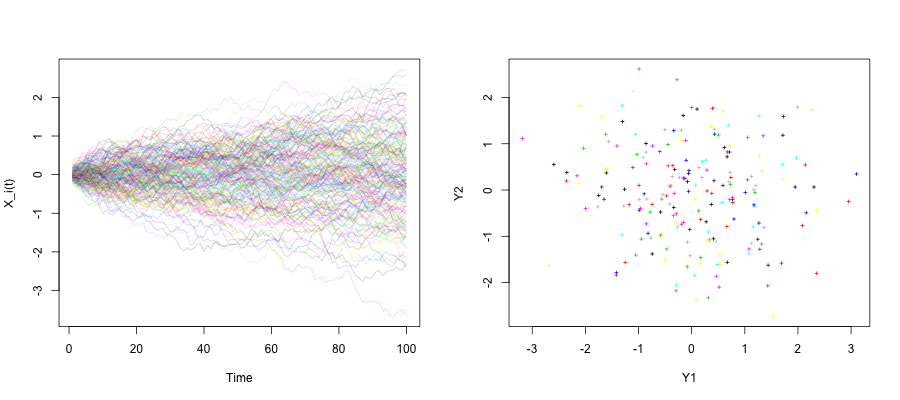}
\end{center}
\caption{Sample of 200 simulated couples of observations $(X_i, {\bf Y}_i)_{i=1,\dots,200}$. The left box contains the covariates $X_i$ and in the right one we present their associated vectors ${\bf Y}_i$. }
\label{example}
\end{figure}

\noindent We have simulated $n=200, 700$ independent realizations $(X_i,{\bf Y}_i)$, $i=1, \dots,n.$ To deal with the Brownian random functions $X_i(t)$, their sample were discretized by 100 points equispaced in $[0,1]$. In Figure \ref{example},  we plot a 200 simulated couples $(X_i,{\bf Y}_i)_{i=1,\dots,200}$ as described above. The left box contains the covariates $X_i$ and in the right one we present the associated vectors ${\bf Y}_i=(Y_{i}^1,Y_{i}^2)$.

We aim to assess, for a fixed curve $X=x$, the performance of the asymptotic conditional confidence ellipsoid given by (\ref{ellipse}) in finite sample. For that we have first to estimate $\mu(x)$. Three parameters should be fixed in this step: the kernel $K$, the bandwidth $h$ and the semimetric $d(\cdot, \cdot)$ which measure the similarity between curves.

\noindent {\bf Choice of the kernel}:  there are many possible density kernel functions. Specialists in non-parametric estimation agree that the exact form of the kernel function does not greatly affect the final estimate with regard to the choice of the bandwidth. In this section, the so-called Gaussian kernel will be used, which is defined by $K(u) = (2\pi)^{-1/2} \exp(-u^2/2)$, for $u\in \ER$.

\noindent {\bf Choice of the bandwidth $h_{n}$}: the bandwidth determines the smoothness of the estimator. The problem of the choice of the bandwidth has been widely studies in non-parametric literature. Recently \cite{rachdi} have proposed a data-driven criterion for choosing this smoothing parameter. The proposed criterion can be formulated in terms of a functional version of cross-validation ideas. \cite{antoniadis} treated the same problem in the context of time series prediction. In the following, the bandwidth $h_{n}$ is selected by $L_{1}$ cross-validation method:
\begin{eqnarray}
\widehat{h}_{n, opt} = \arg\min_{h >0} \sum_{i=1}^n \left\| {\bf Y}_{i} - \widehat{\mu}_{(-i)}(x_{i})\right\|.
\end{eqnarray}

\noindent {\bf Choice of the semi-metric $d(\cdot, \cdot)$}: because of the roughness of our covariate curves we chose a semi-metric computed with the functional principal components analysis with dimension $q=2$.

\begin{figure}[ht!]
\begin{center}
\includegraphics[height=6.5cm,width=12cm]{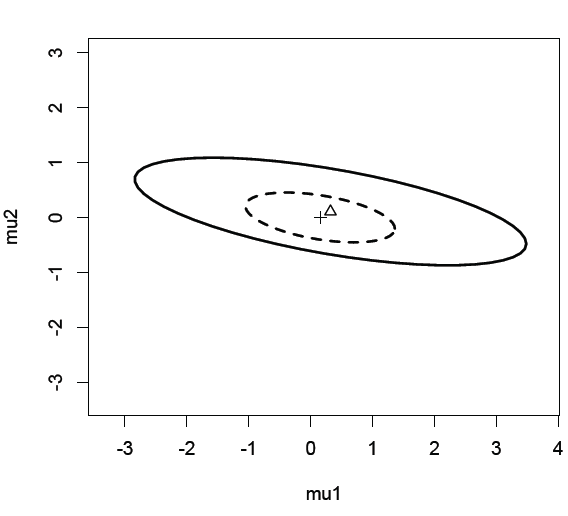}
\end{center}
\caption{Confidence ellipsoid of $\mu(x)$ when $n=200$ (solid lines) and $n=700$ (dashed lines); the centers of the ellipses at $(\mu_{n}^1(x), \mu_{n}^2(x))$ are denoted by triangle (n=200) and cross (n=700).}
\label{conf}
\end{figure}

In Figure \ref{conf}, we plot the $95\%$ confidence ellipses of $\mu(x)$ when $x=0_{\mathcal{F}}$. We can remark from Figure \ref{conf} that the lengths of the major and the minor axes of the confidence ellipse decrease when the sample size $n$ increases. Similar results were obtained for other sample sizes $n$ and values of the curve $x.$

\subsection{Application to Chemiometrical data prediction}
The purpose of this section is to apply our method based on multivariate $L_1$-median regression to some chemiometrical real data and to compare our results to those obtained by other definitions of conditional median studied in literature. For that, we used a sample of spectrometric data available on the web site: http://lib.stat.cmu.edu/datasets/tecator. We have a sample of $n=215$ pieces of meat and for each unit $i$, we observe one spectrometric discretized curve $X_i(\lambda)$ which corresponds to the absorbance measured at a grid of 100 wavelengths (i.e. $X_i(\lambda) = (X_i(\lambda_1), X_i(\lambda_2), \dots, X_i(\lambda_{100}))$).  Figure (\ref{w}) plots the spectrometric curves. Moreover, for each unit $i$, we have at hand its  Moisture content $(Y^1)$, Fat content $(Y^2)$ and Protein content $(Y^3)$ obtained by analytical chemical processing.

\begin{figure}[ht!]
\begin{center}
\includegraphics[height=6.5cm,width=14cm]{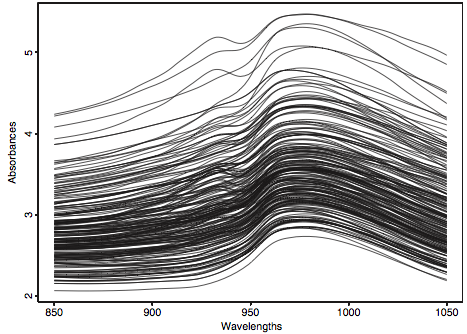}
\end{center}
\caption{The 215 spectrometric curves.}
\label{w}
\end{figure}
\begin{figure}[ht!]
\begin{center}
\includegraphics[height=6cm,width=15cm]{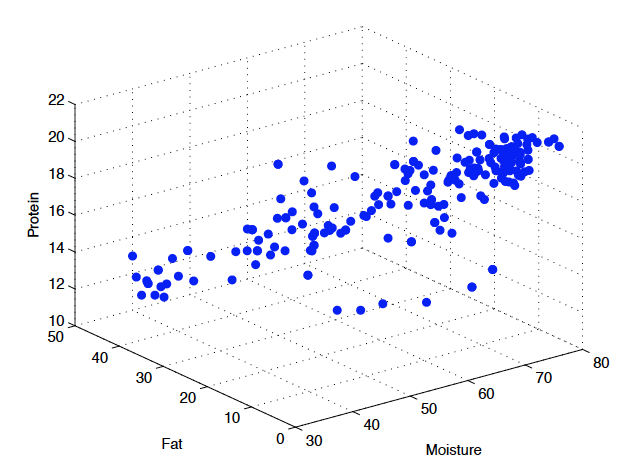}
\end{center}
\caption{The sample of 215 piece of meat.}
\label{vectY}
\end{figure}

Let us denote by ${\bf Y}=(moisture, fat, protein)^t
:=(Y^1,Y^2,Y^3)^t$ the vector of specific chemical contents of meat. Given a new spectrometric curve $X_{new}(\lambda)$, our purpose is to
predict simultaneously the corresponding vector of chemical contents
$\widehat {\bf Y}$ using the multivariate $L_1$-median regression.
Obtaining a spectrometric curve is less expensive (in terms of time and cost) than analytical chemistry needed for determining the percentage of chemical contents. So, it is an important economic challenge to predict the hole vector ${\bf Y}$ from the spectrometric curve.

Let us consider 215 observations $(X_1(\lambda),{\bf Y}_1), \dots, (X_{215}(\lambda), {\bf Y}_{215})$ split into two samples: learning sample (160 observations) and test sample (55 observations). We compare the following three methods, based on multivariate conditional median, to predict the vector of chemical contents ${\bf Y}$ of the test sample. In the following three approaches, we choose the quadratic kernel $K$ defined by:
\begin{eqnarray*}
K(u) = \frac{3}{2}(1-u^2) \1_{[0,1]}.
\end{eqnarray*}
\begin{itemize}
\item[$(i)$] {\bf Non-functional approach (NF)}
\end{itemize}
This method is based on the definition of conditional spatial median studied by \cite{gannoun20032} and \cite{cheng}. This approach does not consider the covariate $X$ as a function but a vector of dimension 100 while the response variable ${\bf Y}$ is a vector.
For each $i=1,\dots, 160$ in the learning sample, the $i^{th}$ vector ${\bf Y}_i$ is predicted as follow:
$$
\widehat {\bf Y}_i= \widehat \mu^{NF}(X_i),
$$
where
$$
\widehat\mu^{NF}(X_i) = \argmin_{u\in\ER^3} \sum_{j=1}^{160} w_{n,j}^{NF}(X_i) \|Y_j-u\|,
$$
\noindent and $w_{n,j}^{NF}(X_i) = \displaystyle K\left(\frac{X_i - X_j}{h_n}\right)\Big/\sum_{j=1}^nK\left(\frac{X_i - X_j}{h_n}\right)$ are the so-called Nadaraya-Watson weights. For the choice of the bandwidth $h_{n}$, \cite{cheng}  gave the exact expression of the optimal bandwidth that minimizes the asymptotic mean square error. In this case $h_{n}$ is of the rate $n^{(-1/104+\epsilon)}$, where $\epsilon>0$ is a sufficiently small constant.

\begin{itemize}
\item[$(ii)$] {\bf Vector Coordinate Conditional Median (VCCM)}
\end{itemize}
This approach supposes that the covariate $X$ is considered as functional. For each $i=1,\dots, 160$ in the learning sample, we predict each component of its vector response ${\bf Y}_i$ by the one-dimensional conditional median. Then we obtain the vector of coordinate conditional medians (VCCMs) defined as
$$\widehat {\bf Y}_i = (\widehat \mu^1(X_i), \widehat \mu^2(X_i), \widehat \mu^3(X_i)),$$

\noindent where each component $\widehat \mu^j(X_i) = (\widehat{F}^j)^{-1}(1/2 \mid X_i)$ is the one-dimensional conditional median estimator.

$\widehat F^j (\cdot \mid X_i)$ is the conditional distribution function estimator of the component $Y^j$ given $X=X_i$. \cite{ferraty2006}, p. 56, have proposed a Nadaraya-Watson kernel estimator of  the conditional distribution, $F^j (\cdot \mid X=X_i)$, when covariate takes values in some infinite dimensional space.  This estimator is given by
$$
\widehat F^j(y^j \mid X=X_i) = \sum_{k=1}^{160} \1_{\{Y_{k}^j \leq y^j\}} K(d(X_i,X_k)/h_n) \Big/ \sum_{k=1}^{160} K(d(X_i,X_k)/h_n) , \quad y^j \in \ER.
$$
\noindent To apply this approach, we used the Ferraty and Vieu's R/routine {\it funopare.quantile.lcv}\footnote{Available at the website www.lsp.ups-tlse.fr/staph/npfda.} to estimate $\widehat \mu^j(X_i).$ The optimal bandwidth is chosen by the cross-validation method on the $k$ nearest neighbours (see \cite{ferraty2006}, p.102 for more details).

\begin{itemize}
\item[$(iii)$] {\bf Conditional Multivariate Median (CMM)}
\end{itemize}
The approach that we propose here supposes the covariate $X$ is a curve and the response $Y$ is a vector.
For each $i=1,\dots, 160$ in the learning sample we take $$\widehat {\bf Y}_i= \widehat \mu(X_i),$$
where
\begin{eqnarray}\label{spatial}
\widehat \mu(X_i)= \argmin_{u\in\ER^3} \sum_{j=1}^{160} w_{n,j}(X_i) \|Y_j-u\|.
\end{eqnarray}

\noindent To estimate the conditional multivariate median, $\widehat\mu(X_i)$, we have adapted the algorithm proposed by \cite{vardi} to the conditional case and used the function \texttt{spatial.median} from the R package \texttt{ICSNP}. As in the previous approach, the optimal bandwidth is chosen by the cross-validation method on the $k$ nearest neighbours.\\

\noindent {\bf A common evaluation procedure}:\\

 We have adapted, to the multivariate case, the algorithm proposed by \cite{attouch} and \cite{ferraty2006}, p.103) in order to get the optimal smoothing parameter $h_n$ for each $X_i$ in the test sample.

\begin{itemize}
\item[$Step1.$] We compute the kernel estimator $\widehat\mu(X_j)$ (resp. $\widehat \mu^k(X_j)$), for all $j$ by using the training sample.
\item[$Step2.$] For each $X_i$ in the test sample, we set $i_\star = \argmin_{j=1,\dots,160} d(X_i,X_j).$
\item[$Step3.$] For each $i=161, \dots,215$, we take
$$\widehat\mu(X_i) = \widehat\mu(X_{i\star}) \quad\mbox{and} \quad \widehat\mu^k(X_i) = \widehat\mu^k(X_{i\star}).$$
\end{itemize}

The used bandwidth for each curve $X_{i}$ in the test sample is the one obtained for the nearest curve in the learning sample. Because the spectrometric curves presented in Figure (\ref{w}) are very smooth, we can choose as semi-metric $d(\cdot,\cdot)$ the $L_2$ distance between the second derivative of the curves. This choice has been made by \cite{attouch} and \cite{ferraty2007} for the same spectrometric curves.

\noindent Both (CMM) and (NF) methods take into account the covariance structure between variables of  of the vector ${\bf Y}$. In fact, the correlation coefficients between $Y_1=moisture$, $Y_2=fat$ and $Y_3=protein$ are given by $\rho_{1,2} = -0.988$, $\rho_{1,3} = 0.814$ and $\rho_{2,3} = -0.860$. As we can see moisture, fat and protein contents in meat are strongly correlated then it will be more appropriate to predict these variables simultaneously rather than each one separately. \\
To compare (CMM), (NF) and (VCCM) methods, we are based on the following criterias:

\begin{itemize}
\item The Absolute Error (AE) gives idea about the prediction of each component of ${\bf Y}$
$$AE_i^j = |Y_i^j - \widehat{\mathcal{L}}^j(X_i)|, \quad\quad \forall i=161,\dots,215\quad \mbox{and} \;\; j=1,2,3.$$
\item A global criteria ($R$) gives idea about error made to predict the vector ${\bf Y}_i$ (for $i=161,\dots, 215$)
 $$R(\widehat {\bf Y}_i)= \|{\bf Y}_i- \widehat{\mathcal{L}}(X_i)\|_{Eucl}$$
\end{itemize}
where $\widehat{\mathcal{L}} := (\widehat{\mathcal{L}}^1, \widehat{\mathcal{L}}^2, \widehat{\mathcal{L}}^3)^t$ represents the estimator of each component of the vector ${\bf Y}$ obtained by (VCCM), (NF) or (CMM) method.\\

\begin{table}
\begin{center}
\begin{tabular}{|l|cccc||cccc||cccc|}
\hline
 & \multicolumn{4}{c}{CMM} & \multicolumn{4}{c}{VCCM} &  \multicolumn{4}{c|}{NF} \\
\cline{2-13}
 & Mean & $Q_{0.25}$ & $Q_{0.5}$ & $Q_{0.75}$ & Mean & $Q_{0.25}$ & $Q_{0.5}$ & $Q_{0.75}$ & Mean & $Q_{0.25}$ & $Q_{0.5}$ & $Q_{0.75}$  \\
\hline
Moist. & 1.301 & 0.479 & 1.100 & 2.202 & 1.776 & 0.460 & 1.879 & 2.383& 7.222 & 1.663 & 6.374& 11.44\\
Fat & 1.565 & 0.430 & 1.500 & 2.401 & 2.343 & 0.925 & 1.716 & 2.867 & 9.758 & 2.328 & 8.4 & 15.24\\
Prot. & 1.125 & 0.300 & 0.800 & 1.437 & 1.313 & 0.518 & 1.182 & 1.806 & 2.446 & 0.787 & 2.329 & 3.394 \\
$R(\widehat Y)$ & 2.638 & 1.349 & 2.530 & 3.623 & 3.561 & 1.877 & 2.909 & 3.799 & 12.6 & 3.523 & 10.6 & 19.27 \\
\hline
\end{tabular}
\caption{Distribution of absolute errors for Moisture, Fat and Protein and global estimation error of the vector ${\bf Y}$.}
\label{table}
\end{center}
\end{table}
\noindent We can conclude from table \ref{table} that our method is more appropriate to predict meat components than (VCCM). In fact, the (VCCM) approach predicts each component of ${\bf Y}$ separately using conditional univariate median. This method supposes independence of the components of ${\bf Y}$ and doesn't take into account the correlation structure between variables. The Non-Functional approach gives the most important prediction errors and this is because of the dimension of the covariate (100 in this case). This problem is well-known in nonparametric estimation as curse of dimensionality. Taking into account the functional aspect of the covariate seems to be necessary in such case.

\section{Concluding remarks}
In this paper, we have introduced a kernel-based estimator for the $L_1$-median of a
multivariate conditional distribution when covariates take values in
an infinite-dimensional space. Prediction using the least square estimates of regression parameters is highly sensitive to outlying points. Therefore, there is no doubt that conditional $L_1$-median can be used to make prediction. We have shown that our estimator is well adapted to predict a multivariate response vector. In fact, in contrast to the
Vector Coordinate Conditional Median method, the multivariate conditional $L_1$-median takes into
account the inter-dependance of the coordinates of the response
vector. Asymptotic results, i.e., almost sure consistency and asymptotic normality, has been given under some regularity conditions. Many extensions can be given to this work. For instance, the same type of theoretical results could be obtained in a non-independence framework (e.g. mixing dependence). Furthermore, it is well known that quantiles are very useful tools to detect outliers and to modelize the dependence of the covariates in lower and upper tails of the response distribution. In future work, we aim to generalize our study to the multivariate quantiles regression when covariates take values in some infinite dimensional space. 
\section*{Appendix: Proofs}

\noindent In order to prove  our results we have to introduce some
further notations. Let
$$\overline G^x_{n, 2}(u) = \mathbb{E}\left( G_{n,2}^x (u)\right) := \frac{1}{\mathbb{E}\Delta_1(x)} \mathbb{E}\left[\|Y_1 - u\|\Delta_1(x)\right],$$
and define the bias of $G^x_n(u)$  as
$$B_n^x(u) = \overline G^x_{n, 2}(u) - G^x(u).$$
\noindent Consider now the following quantities
$$R_n^x(u) = - B_n^x(u) \left(G^x_{n,1} - 1\right)$$
and
$$Q_n^x(u) = \left(G^x_{n,2}(u)  - \overline G_{n, 2}^x(u) \right)- G^x(u)\left(G^x_{n,1}  - 1 \right).$$
It is then clear that the following decomposition holds

\begin{equation}\label{decomp1}
G^x_n(u) - G^x(u) = B_n^x(u) + \frac{R_n^x(u) + Q_n^x(u) }{
G^x_{n,1}}.
\end{equation}

Since $ G^x_{n,1}$ is independent of $u$, it follows from
decomposition  (\ref{decomp1})  that

\begin{equation}\label{decomp2} %
\sup_{u\in \mathbb{R}^d}|G^x_n(u) - G^x(u)|  \leq \sup_{u\in
\mathbb{R}^d}|B_n^x(u)| + \frac{ \sup_{u\in
\mathbb{R}^d}|R_n^x(u)| + \sup_{u\in \mathbb{R}^d}|Q_n^x(u)| }{
G^x_{n,1}}.
\end{equation}

The proof of Proposition \ref{resG} is split up into several
lemmas,
 given hereafter, establishing respectively the convergence almost surely (a.s.)  of $G^x_{n,1}$ to
 $1$ and that of $B_n^x(u)$, $R_n^x(u)$
  and $Q_n^x(u)$ (with rate) to zero.

  We start by the following technical lemma whose  proof my be found in
Ferraty et {\it al.} (2007).

\begin{lemma}\label{lemma1}
Assume that conditions (H1),(H2)  hold true. For any real numbers
$j \geq 1$ and $k \geq 1,$ as $n \rightarrow \infty$, we have

\begin{itemize}
\item[$(i)$] $\displaystyle\frac{1}{\phi(h)} \mathbb{E}[\Delta_1^j(x)] = M_j g(x) + o(1)$
\item[$(ii)$] $\displaystyle \frac{1}{\phi^k(h)} (\mathbb{E}(\Delta_1(x)))^k = M_1^k g^k(x) + o(1).$
\end{itemize}
\end{lemma}

%

Lemma below gives the convergence rate of the quantity
$G^x_{n,1}$.

\begin{lemma}\label{lemma2}
Under assumptions (H1)-(H2) and condition (\ref{CondThm1})(i), we
have
 $$G^x_{n,1} - 1 = O_{a.s}\left(\sqrt{\displaystyle\frac{\log n}{n\phi(h)}} \right).$$
\end{lemma}

\begin{proof} of Lemma \ref{lemma2}. Let us denote by
$$R^x_{n,1} = G^x_{n,1} - 1 := \frac{1}{n} \sum_{i=1}^n L_{n,i}(x),$$
where $L_{n,i}(x) = \Delta_i^{\star}(x)-\mathbb{E}(\Delta_i^{\star}(x))$ and $\Delta_i^{\star}(x) = \frac{\Delta_{i}(x)}{\mathbb{E}(\Delta_1(x))}$.
To apply the exponential inequality given by Corollary A.8(i) of Ferraty and Vieu (2006) in Appendix A  we have first to show that for all $m\geq 2$ there exist a positive constant $C_{m}$ such that $\mathbb{E}|L_{n,1}^{m}(x)| \leq C_{m}a^{2(m-1)}$. We have
\begin{eqnarray*}
\mathbb{E}\left(|L_{n,1}(x)|^m \right)\leq C \sum_{k=0}^m
\displaystyle{m \choose k} \mathbb{E}
\left[\left(\Delta_1^\star(x)\right)^k\right]
\left[\mathbb{E}(\Delta_1^\star(x)) \right]^{m-k}.
\end{eqnarray*}
Then using Lemma $\ref{lemma1}$ we get $\mathbb{E}\left(|L_{n,1}(x)|^m \right)\leq C_m \max_{k=0,1,\dots,m} (\phi(h))^{1-k} \leq C_{m} (\phi(h))^{1-m}.$ Therefore, we have $a^{2}=(\phi(h))^{-1}.$
Now, for all $\epsilon>0$, we have
\begin{eqnarray*}
\mathbb{P}(|R^x_{n,1}|>\epsilon) &\leq& 2 \exp\left\{- \frac{n\epsilon^2}{2\phi(h)(1+\epsilon)} \right\}.
\end{eqnarray*}
The desired result follows from Borel Cantelli Lemma by choosing
$\epsilon=\epsilon_0\sqrt{\log n/n\phi(h)}$ where $\epsilon_0$ is
a large enough positive  constant.
\end{proof}

The following lemma describes the uniform asymptotic behavior of
the conditional bias term $B_n^x(u)$ as well as that of $R_n^x(u)$
and $Q_n^x(u)$ with respect to $u$.

\begin{lemma}\label{lemma3}
(i) Under conditions  (H1)-(H2) (H3)(i), we have
\begin{equation}\label{rateB}
\sup_{u\in \mathbb{R}^d}|B_n^x(u)| = O_{a.s.}(h^\beta).
\end{equation}

(ii) If in addition that  (H1)-(H2) hold true and condition
(\ref{CondThm1}) is satisfied, we have
\begin{equation}\label{rateR}
\sup_{u\in \mathbb{R}^d}|R_n^x(u)| =
O_{a.s.}\left(h^\beta\sqrt{\frac{\log n}{n\phi(h)}}\right)
\end{equation}
\end{lemma}
\begin{proof} of Lemma \ref{lemma3}. Recall that
$$B_n^x(u) = \overline G^x_{n,2}(u) - G^x(u).$$
Conditioning by $X$ and using   the definition of $G^x (u)$ and
condition (H3)(i), one has
\begin{eqnarray*}
|B_n^x(u)| &=& \left| \frac{1}{\mathbb{E}\Delta_1(x)}\mathbb{E}\left\{\Delta_1(x) \mathbb{E}[\|Y_1-u\|\;|\;X]\right\} - G^x(u) \right|\\
&=& \left| \frac{1}{\mathbb{E}\Delta_1(x)}\mathbb{E}\left\{\Delta_1(x) (G^X(u) - G^x(u)) \right\} \right| \\
&\leq & \sup_{x' \in B(x,h)}|G^{x'}(u) - G^x(u)| = O_{a.s.}(h^\beta).
\end{eqnarray*}
The later quantity is independent of $u$, this leads to
$\sup_{u\in \mathbb{R}^d}|B_n^x(u)|=O_{a.s.}(h^\beta).$

 \noindent
Now, to deal with the quantity $R_n^x(u),$ write it as $R_n^x(u) =
- B_n^x(u)\left( G^x_{n,1} - 1\right).$  Therefore
$$\sup_{u\in \mathbb{R}^d}|R_n^x(u)| = \sup_{u\in \mathbb{R}^d}| B_n^x(u)| | G^x_{n,1} - 1|.$$
The statement (\ref{rateR}) follows from (\ref{rateB}) combined
with Lemma \ref{lemma2}.

\end{proof}


\begin{lemma}\label{lemma4}
 Under assumptions (H1)-(H2), (H4)(i),  conditions
(\ref{CondThm1}) and (\ref{Cond3.Thm1}) we have
$$\sup_{u\in \mathbb{R}^d}|G^x_{n,2}(u) - \overline G^x_{n,2}(u)|= O_{a.s.}\left(\sqrt{\frac{\log n}{n\phi(h)}}\right).$$
\end{lemma}

\begin{proof} of Lemma \ref{lemma4}. For $u\in \ER^d$ and $r>0$, let
$$S(u, r)=\{ u^\prime: u^\prime\in\mathbb{R}^d, ||u^\prime-u||\leq r\},$$ be the sphere of radius $r$ centered at $u$.
Let $[-n^\gamma,   n^\gamma]^d$, for $1/2<\gamma<2$, be an
interval of $\mathbb{R}^d$.  Divide  $[-n^\gamma, n^\gamma]$ into
$k_n$ subintervals each of length $b_n=[2n^\gamma/k_n]$ (where
$[t]$ is the integer part of $t$). Since the set $S(0,
n^\gamma)=\{ u^\prime: ||u^\prime||\leq
 n^\gamma \}$ is compact, it can be covered by $k_n^d$ bounded hypercubes of the form
 $$S_{n,j}:=S(u_j, b_n)=\{ u^\prime: ||u^\prime-u_j||\leq b_n\}, \quad j=1, \ldots, k_n^d.$$
%

We have
\begin{eqnarray}\label{Sup1}
& &\sup_{||u||\leq n^\gamma}| G^x_{n,2}(u)-
\overline G^x_{n,2}(u) |\nonumber \\
&\leq & \max_{1\leq j\leq k_n^d} \sup_{u\in S_{n,j}}| G^x_{n,2}(u)
- G^x_{n,2}(u_j) | +
\max_{1\leq j\leq k_n^d} | G^x_{n,2}(u_j) -
\overline{G}^x_{n,2}(u_j) | \nonumber \\
&+&
 \max_{1\leq j\leq k_n^d} \sup_{u\in S_{n,j}}|
\overline{G}^x_{n,2}(u) - \overline{G}^x_{n,2}(u_j) |
:= I_{n,1}+I_{n,2}+I_{n,3}.
\end{eqnarray}
Observe now that
\begin{eqnarray*}\sup_{u\in S_{n,j}}| G^x_{n,2}(u) - G^x_{n,2}(u_j)
| &\leq& \frac{1}{n\mathbb{E}(\Delta_1(x))} \sum_{i=1}^n
\sup_{u\in S_{n,j}}\Big| ||Y_i-u||-||Y_i-u_j||\Big|\, \Delta_i(x) \\
 &\leq& \frac{1}{n\mathbb{E}(\Delta_1(x))} \sum_{i=1}^n
 \Delta_i(x) \sup_{u\in S_{n,j}}||u -u_j|| =b_n G_{n,1}^x,
 \end{eqnarray*}
and
\begin{eqnarray*}\sup_{u\in S_{n,j}}| \overline{G}^x_{n,2}(u) - \overline{G}^x_{n,2}(u_j)
| &\leq& \mathbb{E}\left[ \sup_{u\in S_{n,j}}| G^x_{n,2}(u) -
G^x_{n,2}(u_j)|\right]=b_n.
 \end{eqnarray*}
 If we denote by $\alpha_n=\sqrt{n\phi(h)/\log n}$ the convergence rate,  one gets by Lemma \ref{lemma2}
$$\alpha_n(I_{n,1}+I_{n,3})=\mathcal{O}(\alpha_n b_n
(1+G_{n,1}^x))=\mathcal{O}(\alpha_n b_n)=\mathcal{O}(\alpha_n n^\gamma/k_n^d).$$
The choice of $k_n^d=[\alpha_n n^\gamma \log n]$ implies that
\begin{eqnarray}\label{condition-Vitesse1}
\alpha_n(I_{n,1}+I_{n,3})=o(1).
\end{eqnarray}
In order to evaluate the term $I_{n,2}$,  let us denote by
$$\Delta_{i}^\star(x) = \frac{\Delta_i(x)}{\mathbb{E}\Delta_1(x)},$$
and
$$Z_{n,i}(x) = \|Y_i-u_j\|\Delta_{i}^\star(x) - \mathbb{E}\left[ \|Y_1-u_j\|\Delta_{1}^\star(x)\right].$$
Then, we have
$$G^x_{n,2}(u_j) - \overline G^x_{n,2}(u_j) = \frac{1}{n}\sum_{i=1}^n Z_{n,i}(x).$$

\noindent For all $m\in \mathbb{N}-\{0\},$ observe that
$$Z_{n,i}^m(x) = \sum_{k=0}^m \displaystyle{m \choose k}\left(\|Y_i-u_j\|\Delta_i^\star(x)\right)^k (-1)^{m-k}
\left[\mathbb{E}(\|Y_1-u_j\|\Delta_1^\star(x)) \right]^{m-k}.$$
In order to apply an exponential type inequality, we  have to give
an upper bound for  $\mathbb{E}\left(|Z_{n,1}(x)|^m \right)$. It
follows from the above inequality that
\noindent
\begin{eqnarray*}
\mathbb{E}\left(|Z_{n,1}(x)|^m \right)\leq C \sum_{k=0}^m
\displaystyle{m \choose k} \mathbb{E}
\left[\left(\|Y_1-u_j\|\Delta_1^\star(x)\right)^k\right]
\left[\mathbb{E}(\|Y_1-u_j\|\Delta_1^\star(x)) \right]^{m-k}.
\end{eqnarray*}
 On the other hand, we have  for any $k\geq2$
\begin{eqnarray*}
\mathbb{E}\left[\left(\|Y_1-u_j\|\Delta_1^\star(x)\right)^k\right]
&=& \mathbb{E}\left[(\Delta_1^\star(x))^k \mathbb{E}\left(\|Y_1-u_j\|^k\;|\;X_1 \right) \right]\\
&=&\mathbb{E}\left[(\Delta_1^\star(x))^k G_k^{X_1}(u_j) \right].
\end{eqnarray*}
Using the first part of condition $(H4)(i)$, which implies that
 $G_k^{x}(u_j)$ is bounded uniformly for all $j$,
one may write

\begin{eqnarray*}
\mathbb{E}\left[\left(\|Y_1-u_j\|\Delta_1^\star(x)\right)^k\right]
&\leq& \mathbb{E}\left[(\Delta_1^\star(x))^k |G^{X_1}_k(u_j) -
G^x_k(u_j)|
 \right] + G^x_k(u_j) \mathbb{E}((\Delta_1^\star(x))^k)\\
&\leq& \mathbb{E}((\Delta_1^\star(x))^k) \left[\max_{j}\sup_{x'\in B(x,h)} |G^{x'}_k(u_j) - G^x_k(u_j)| + \max_{j} G^x_k(u_j)\right]\\
&\leq& C_0\mathbb{E}\left[(\Delta_1^\star(x))^k \right],
\end{eqnarray*}
where $C_0$ is a positive constant. Moreover, we have
$\mathbb{E}\left(\|Y_1-u_j\| \Delta_1^\star(x)
\right)=\mathcal{O}(1)$ uniformly in $j$  since
$\mathbb{E}\left[\Delta_1^\star(x)\right]=1$ and $\sup_u
\mathbb{E}(|| Y_1-u|| \;| X)<\infty$ in view of condition (\ref{G}).

 Therefore
$[\mathbb{E}\left(\|Y_1-u_j\| \Delta_1^\star(x)
\right)]^{m-k}=\mathcal{O}(1)$.

\noindent  Next, applying  Lemma $\ref{lemma1}$, one may write
\begin{eqnarray*}
\mathbb{E}\left[(\Delta_1^\star(x))^k\right] 
&=&(\phi(h))^{1-k} \left[\frac{M_k}{M_1^k}g^{1-k}(x) + o(1) \right].
\end{eqnarray*}
\noindent Thus
\begin{eqnarray*}
\mathbb{E}\left(|Z_{n,1}(x)|^m \right)&\leq& C_m \max_{k=0,1,\dots,m} (\phi(h))^{1-k}
\end{eqnarray*}

\noindent where $C_m$ is a real positive constant depending on $m$.
Because $\phi(h)$ tends to zero as $n$ goes to infinity, it
comes that

$$\mathbb{E}\left(|Z_{n,1}(x)|^m \right) = \mathcal{O}\left( (\phi(h))^{1-m}\right).$$

\noindent Now, applying Corollary $A.8-i$ in Ferraty \& Vieu
(2006) $k_n^d$ times with $a^2 = (\phi(h))^{-1}$ we obtain, by
choosing
$$\epsilon=\epsilon_n=3\epsilon_0\sqrt{v_n}  \quad \mbox{where} \quad  v_n = (a^2\log
n)/n = \log n/(n\phi(h)) \longrightarrow 0 \quad \mbox{as} \ n\to
\infty,$$ that
\begin{eqnarray*}\label{Sup4}
%
\mathbb{P}\left( | I_{n,2} | \geq \epsilon \right) &\leq& 2k_n^d
\exp\left( -\epsilon_0^2 \log n \left[ \frac{1}{2(1+\epsilon_0
\sqrt{v_n})} \right] \right)\leq 2k_n^d n^{-\epsilon_0^2}.
\end{eqnarray*}
One may choose $\epsilon_0$ large enough such that
$$\sum_{n} \mathbb{P}\left( | I_{n,2} | \geq \epsilon
\right)<\infty.$$
We conclude by Borel-Cantelli lemma and
(\ref{condition-Vitesse1}) that
$$\alpha_n\sup_{||u||\leq
n^\gamma}| G^x_{n,2}(u)- \overline
G^x_{n,2}(u)|=\mathcal{O}_{a.s}(\alpha_n\sqrt{v_n})=\mathcal{O}_{a.s}(1).$$


Next, we have
\begin{eqnarray*}\label{vitesse2}
\sup_{u\in \mathbb{R}^d}\alpha_n|G^x_{n,2}(u)-\overline
G^x_{n,2}(u)|
&\leq& \sup_{||u|| \leq n^\gamma}\alpha_n|G^x_{n,2}(u)-\overline
G^x_{n,2}(u)|+ \sup_{||u|| >
n^\gamma}\alpha_n|G^x_{n,2}(u)-\overline G^x_{n,2}(u)|\nonumber \\
&=&
\sup_{||u|| > n^\gamma}\alpha_n|G^x_{n,2}(u)-\overline
G^x_{n,2}(u)| + \mathcal{O}_{a.s.}(1),
\end{eqnarray*}
in view of the above result. Now, we have
\begin{eqnarray}\label{ConUniform1}
& &\alpha_n\sup_{u: ||u||\geq n^\gamma}| G^x_{n,2}(u) -
\overline G^x_{n,2}(u)|\nonumber \\
&\leq &
\alpha_n\sup_{u: ||u||\geq n^\gamma}|   G^x_{n,2}(u) |
+\alpha_n\sup_{u: ||u||\geq n^\gamma}| G^x(u) |
+ \alpha_n\sup_{u}|  G^x(u)- \overline{G}^x_{n,2}(u) |.
\end{eqnarray}

The last term in (\ref{ConUniform1}) is zero for large $n$,
since conditioning by $X$, one may write
\begin{eqnarray*}
\alpha_n|\overline{G}_{n,2}^x(u)-G^x(u)| &=& \alpha_n
|B_n^x(u)|=\mathcal{O}_{a.s.}(h_n^{\beta}\alpha_n)=_{a.s.}(1)
\end{eqnarray*}
in view  Lemma \ref{lemma3} (i)  whenever condition
(\ref{CondThm1})(ii) is satisfied.
%
For the second term in (\ref{ConUniform1}), we have

$$\alpha_n \sup_{||u||>n^\gamma}G^x(u)\leq \frac{\alpha_n}{n^\gamma}
\sup_{||u||>n^\gamma}||u||G^x(u)= o(1),$$ whenever  $\gamma >1/2$
and the condition (\ref{Cond3.Thm1}) is satisfied.

Moreover, we have for any $\epsilon>0$
\begin{eqnarray*}
 & & \mathbb{P}\left\{\alpha_n \sup_{u:||u||\geq n^\gamma}| G^x_{n,2}(u) |\geq
\epsilon \right\} \\
 &\leq & \mathbb{P}\left\{\alpha_n \sup_{u:||u||\geq
n^\gamma}\frac{1}{n\mathbb{E}(\Delta_1)} \sum_{i:
||Y_i-u||>n^\gamma/2}
||Y_i-u|| \Delta_i(x)||\geq \epsilon/2 \right\} \nonumber\\
&+& \mathbb{P}\left\{\alpha_n \sup_{u:||u||\geq
n^\gamma}\frac{1}{n\mathbb{E}(\Delta_1)} \sum_{i: ||Y_i-u|| \leq
n^\gamma/2} ||Y_i-u|| \Delta_i(x)||\geq \epsilon/2
\right\}:=J_{n,1}+ J_{n,2}.
\end{eqnarray*}

To  treat  $J_{n,1}$, denote by
$$A_n(\omega):=\{ \omega: \alpha_n
\sup_{||u||>n^\gamma} \frac{1}{n}\sum_{i=1:
||Y_i-u||>n^{\gamma}/2}^n ||Y_i-u||\Delta_i \geq \epsilon/2 \}.$$

\noindent  The event $A_n(\omega)$ is nonempty if and only if
there exists at least $i_0 $ ($1\leq i_0\leq n$)  such that
$||Y_{i_0}-u||>n^\gamma/2$. Thus $"A_n(\omega) \neq \varnothing"
\subset \cup_{i=1}^n \{\omega: ||Y_i-u||\geq n^\gamma/2\}$. It
follows from Markov's inequality, if
$\mathbb{E}(||Y_1-u||)<\infty$, that

$$\mathbb{P}\left( A_n(\omega)\neq \varnothing \right)
=\mathcal{O}(n^{-(\gamma-1)})\quad \mbox{and} \quad  \sum_n
\mathbb{P}\left( A_n(\omega)\neq \varnothing \right)<\infty,$$
whenever $\gamma>1$, which implies that $J_{n,1}=o_{a.s.}(1)$ by
Borel-Cantelli Lemma.

To deal with  $J_{n,2}$, let us denote by

 $$B_n(\omega):=\{\omega: \alpha_n \sup_{u:||u||\geq
n^\gamma}\frac{1}{n\mathbb{E}(\Delta_1)} \sum_{i: ||Y_i-u|| \leq
n^\gamma/2} ||Y_i-u|| \Delta_i(x)||\geq \epsilon/2 \}.$$

$B_n(\omega)$  is  nonempty if and only if there exists at least
$i_0 $ ($1\leq i_0\leq n$)  such that $||Y_{i_0}-u|| \leq
n^\gamma/2$. The later inequality   implies that $||Y_{i_0}-u||
-||u||\leq 0$ whenever $||u|| \geq n^\gamma$. Moreover, we have
(by triangle inequality), whenever the above conditions are hold,
that
$$||Y_{i_0}||\geq \Big| ||Y_{i_0}-u||- ||u|| \Big| =-||Y_{i_0}-u||+||u||> n^\gamma/2.$$
Therefore,
$$"B_n(\omega)\neq \varnothing " \subset \{ \exists  i_0: 1\leq i_0\leq
1, \ ||Y_{i_0}|| \geq n^{\gamma}/2 \}.$$

We conclude  as above that $J_{n,2}=o_{a.s.}(1)$ whenever
$E(||Y_1||)>\infty$ and $\gamma>1$.

This ends the proof of Lemma \ref{lemma4}.
\end{proof}

\begin{lemma}\label{lemmaQ_n}
 Under assumptions (H1)-(H2), (H4)(i) and condition
(\ref{CondThm1})(i),  we have
\begin{equation}\label{rateQ}
Q_n^x(u) = \mathcal{O}_{a.s.}\left(\left(\frac{\log n}{n\phi(h)}
\right)^{1/2}\right).
\end{equation}
\end{lemma}

\begin{proof} of  Lemma \ref{lemmaQ_n}.
In order to check the  statement  (\ref{rateQ}), recall that
$$Q_n^x(u) = \left(G^x_{n,2}(u) - \overline G^x_{n,2}(u)\right) - G^x(u) \left(G^x_{n,1} - 1 \right).$$
The  result follows then  from Lemmas \ref{lemma2} and
\ref{lemma4}.
\end{proof}

\begin{proof} of Proposition \ref{resG}. The proof follows from
Lemmas \ref{lemma2}, \ref{lemma3}, \ref{lemma4} and
\ref{lemmaQ_n}.

\end{proof}

\begin{proof} of Theorem \ref{thm2}.

We have  from the definitions of $\mu(x)$ and $\mu_n(x)$  and the
existence and the uniqueness of these quantities  that:
$$G^x(\mu(x))=\inf_{u\in \mathbb{R}^d}G^x(u) \quad \mbox{and}
\quad  G_n^x(\mu_n(x))=\inf_{u\in \mathbb{R}^d}G_n^x(u).$$ It
follows then
\begin{eqnarray}
|G^x(\mu(x)) - G^x(\mu_n(x))|
&\leq& |G^x(\mu(x)) - G_n^x(\mu_n)| + |G_n^x(\mu_n(x)) - G^x(\mu_n(x))|\nonumber\\
&=& |-(-\inf_{u\in \mathbb{R}^d}G^x(u) + \inf_{u\in \mathbb{R}^d}G_n^x(u))| + |G_n^x(\mu_n(x)) - G^x(\mu_n(x))|\nonumber\\
&=& |-\sup_{u\in \mathbb{R}^d}G^x(u) + \sup_{u\in \mathbb{R}^d}G_n^x(u)| + |G_n^x(\mu_n(x)) - G^x(\mu_n(x))|\nonumber\\
&\leq&\sup_{u\in \mathbb{R}^d} |G^x(u) - G_n^x(u)| + |G_n^x(\mu_n(x)) - G^x(\mu_n(x))|\nonumber\\
 &\leq& 2 \sup_{u \in\ER^d}|G^x_n(u) - G^x(u)|. \label{sup}
\end{eqnarray}
%
Moreover, since for any fixed $x\in {\cal F}$, the function
$G^x(\cdot)$ is uniformly continuous and because $\mu(x)$ is the
unique minimizer of the function $G^x(\cdot)$, we have then,  for
any $\epsilon >0,$
\begin{eqnarray}\label{uniqueness}
\inf_{u: \|\mu(x) - u\| \geq \epsilon} G^x(u) > G^x(\mu(x)),
\end{eqnarray}
which means that there exists for every $\epsilon >0$, a number
$\eta(\epsilon) >0$ such that $G^x(u) > G^x(\mu(x)) +
\eta(\epsilon)$ for every $u$ such that   $\|\mu(x) - u\|
\geq\epsilon.$
 This implies that the event $\{ \|\mu(x) - \mu_n(x)\| > \epsilon \}$ is included in the event $\{G^x(\mu_n(x)) > G^x(\mu(x)) +\eta(\epsilon) \}.$

\noindent Using inequality (\ref{sup}) we get

\begin{eqnarray*}
\sum_{n\geq 1}\mathbb{P}\left(\|\mu_n(x) - \mu(x)\|  > \epsilon \right)
&\leq& \sum_{n\geq 1}\mathbb{P}\left(G^x(\mu_n(x)) > G^x(\mu(x)) + \eta(\epsilon)  \right)\\
&\leq& \sum_{n\geq 1}\mathbb{P} \left(\sup_{u \in \ER^d} |G_n^x(u)
- G^x(u)| > \eta(\epsilon)/2 \right)<\infty,
\end{eqnarray*}

\noindent similarly to  the proof of the Proposition \ref{resG}.
The statement (\ref{Con.median}) follows then from an application
of Borel-Cantelli Lemma.
\end{proof}


\begin{proof} of Proposition \ref{variance}

\noindent To prove  Proposition \ref{variance}, it suffices to see that

\begin{eqnarray}\label{inegalite}
\|\widetilde H_{n}^x(\xi_n(j)) - H^x(\mu)\| \leq \|\widetilde H_{n}^x(\xi_n(j)) -
\widetilde H_{n}^x(\mu)\|+\|\widetilde H_{n}^x(\mu) - H^x(\mu)\|.
\end{eqnarray}

\noindent Concerning the first term, observe that
\begin{eqnarray}\label{decompH}
\|\widetilde H_{n}^x(\xi_n(i)) - \widetilde H_{n}^x(\mu)\| &\leq& \frac{1}{n\; \mathbb{E}(\Delta_1(x))}
\sum_{i=1}^n \|\mathcal{M} (Y_i,\xi_n(j)) - \mathcal{M}(Y_i,\mu)\| \;\Delta_i(x) \nonumber\\
&:=& \mathcal{A}_{n}+ \mathcal{B}_{n},
\end{eqnarray}
where
$$
\mathcal{A}_{n}:= \frac{\sqrt{d}}{n\mathbb{E}(\Delta_{1}(x))} \sum_{i=1}^n \frac{\Big| \|Y_{i}- \mu\| - \| Y_{i}-\xi_{n}(j)\| \Big| \Delta_{i}(x)}{\|Y_{i} - \mu \| \; \| Y_{i} - \xi_{n}(j)\|}
$$
and

$$
\mathcal{B}_{n}:= \frac{1}{n\mathbb{E}(\Delta_{1}(x))}\sum_{i=1}^n
\Delta_{i}(x) \frac{\Big|\Big| \| Y_{i} - \xi_{n}(j)\|
\;\mathcal{U}(Y_{i} - \mu)\; \mathcal{U}^T(Y_{i} - \mu) - \| Y_{i}
- \mu\| \;\mathcal{U}(Y_{i} - \xi_{n}(j))\; \mathcal{U}^T(Y_{i} -
\xi_{n}(j)) \Big|\Big|}{\|Y_{i} - \mu \| \; \| Y_{i} -
\xi_{n}(j)\|}.
$$
Using Theorem \ref{thm2} and the triangular inequality we can
easily see that $\mathcal{A}_{n} = o_{a.s.}(1)\times
\frac{1}{n\mathbb{E}(\Delta_{1}(x))}\sum_{i=1}^n
\frac{\Delta_{i}(x)}{\|Y_{i}-\mu \|^2}$.

\noindent Combining Markov and Cauchy-Schwarz inequalities and
making use of the assumption H3-(iii), we can easily prove that
$\frac{1}{n\mathbb{E}(\Delta_{1}(x))}\sum_{i=1}^n
\frac{\Delta_{i}(x)}{\|Y_{i}-\mu \|^2} =
\mathcal{O}_{\mathbb{P}}(1)$. Then we conclude that
$\mathcal{A}_{n} = o_{\mathbb{P}}(1).$

For the second term  $\mathcal{B}_{n}$ of the inequality
(\ref{decompH}), we have by triangular inequality and the fact
that $\|U(Y_{i} - \theta)\|=1$, that
\begin{eqnarray*}
\Big|\Big| \| Y_{i} - \xi_{n}(j)\| \;\mathcal{U}(Y_{i} - \mu)\; \mathcal{U}^T(Y_{i} - \mu) - \| Y_{i} - \mu\| \;\mathcal{U}(Y_{i} - \xi_{n}(j))\; \mathcal{U}^T(Y_{i} - \xi_{n}(j)) \Big|\Big| &\leq &\\
\Big| \| Y_{i}-\xi_{n}(j)\| - \| Y_{i}-\mu\|\Big| +  \|Y_{i}-\mu \|\; \Big|\Big| \mathcal{U}(Y_{i}-\mu)\;\mathcal{U}^{T}(Y_{i}-\mu) - \mathcal{U}(Y_{i}-\xi_{n}(j))\;\mathcal{U}^{T}(Y_{i}-\xi_{n}(j))\Big|\Big| &\leq &\\
\| \mu - \xi_{n}(j) \| + \|Y_{i} - \mu \| \times \Big|\Big|\, \mathcal{U}(Y_{i}-\mu)\; \mathcal{U}^T(Y_{i}-\mu) - \mathcal{U}(Y_{i}-\xi_{n}(j))\; \mathcal{U}^T(Y_{i}-\xi_{n}(j)) \Big|\Big|.
\end{eqnarray*}
\noindent Since
 \begin{eqnarray*}
 \mathcal{U}(Y_{i}-\mu)\, \mathcal{U}^T(Y_{i}-\mu) - \mathcal{U}(Y_{i}-\xi_{n}(j))\, \mathcal{U}^T(Y_{i}-\xi_{n}(j)) &=& \left[\,
 \mathcal{U}(Y_{i}-\mu) - \mathcal{U}(Y_{i}-\xi_{n}(j))\,\right]\, \mathcal{U}^{T}(Y_{i} - \mu)\\
 & & + \;\mathcal{U}(Y_{i}-\xi_{n}(j)) \left[\,\mathcal{U}^T(Y_{i}-\mu) - \mathcal{U}^T(Y_{i}-\xi_{n}(j)) \,\right],
 \end{eqnarray*}
\noindent and $\displaystyle\|\mathcal{U}(Y_{i}-\mu) - \mathcal{U}(Y_{i}-\xi_{n}(j)) \| \leq 2\frac{\| \mu - \xi_{n}(j) \|}{\| Y_{i} -
 \xi_{n}(j)\|}$, we can conclude, by using Theorem \ref{thm2}, that
\begin{eqnarray*}
\mathcal{B}_{n} = o_{a.s.}(1) \times
\frac{1}{n\mathbb{E}(\Delta_{1}(x))}\sum_{i=1}^n
\frac{\Delta_{i}(x)}{\|Y_{i}-\mu \|^2}
\end{eqnarray*}
Finally, using the same arguments as above (concerning the  proof
of the term $\mathcal{A}_{n}$), we get $\mathcal{B}_{n} =
o_{\mathbb{P}}(1)$ and this is allows us to conclude that
$\|\widetilde H_{n}^x(\xi_n(i)) - \widetilde H_{n}^x(\mu)\| =
o_{\mathbb{P}}(1)$.
\noindent Now we are interesting to the second term of the right
side term of (\ref{inegalite}). Write
\begin{eqnarray*}
\widetilde H_{n}^x(\mu) - H^x(\mu) = \underbrace{\widetilde
H_{n}^x(\mu) - \mathbb{E}[\widetilde H_{n}^x(\mu)]}_{K_{n,1}} +
\underbrace{\mathbb{E}[\widetilde H_{n}^x(\mu)] -
H^x(\mu)}_{K_{n,2}}.
\end{eqnarray*}
We have to show that each term $K_{n,i}$  $(i=1, 2)$ is
asymptotically negligible.  We have
\begin{eqnarray*}
\| K_{n,1} \|^2 =tr(K_{n,1}^T K_{n,1}) = \sum_{k=1}^d \sum_{j=1}^d \mathcal{Z}_{k,j}^2
\end{eqnarray*}
where $(\mathcal{Z}_{k,j})_{1\leq k,j \leq d}$ is the general term
of the matrix $K_{n,1}^T K_{n,1}$ which may be  can be written as
\begin{eqnarray*}
\mathcal{Z}_{k,j} &=& \frac{1}{n\mathbb{E}(\Delta_{1}(x))}
\sum_{i=1}^n \left[ \mathcal{M}_{k,j}(Y_{i}, \mu)\Delta_{i}(x) -
\mathbb{E}\left(\mathcal{M}_{k,j}(Y_{i}, \mu)\Delta_{i}(x)\right)
\right].
 \end{eqnarray*}
Using the assumption (H3)-(iv), Lemma \ref{lemma1} and corollary
A.8 of \cite{ferraty2006}, we can easily prove that for all
$1\leq k,j \leq d $, $Z_{k,j} = o_{\mathbb{P}}(1)$.



\noindent To handle $K_{n,2}$,  observe that
\begin{eqnarray*}
\|K_{n,2}\| &=& \left\|\mathbb{E}\left[\frac{\sum_{i=1}^n \mathcal{M}(Y_i,\mu)\;\Delta_i(x)}{n\;\mathbb{E}(\Delta_1(x))}\right] - H^x(\mu)\right\|\\
%
&\leq&\frac{1}{\mathbb{E}(\Delta_1(x))}
 \mathbb{E}\left(\|H^{X_1}(\mu)-H^x(\mu)\| \Delta_1(x)\right) \\
 &\leq& \sup_{x' \in B(x,h)}
\|H^{x'}(\mu)-H^x(\mu)\|=o_{a.s.}(1)
\end{eqnarray*}
in view of  condition $(H3)(ii)$.
\end{proof}
\begin{lemma}\label{l1}
Under hypothesis (H1)-(H2) and (H4)(ii), and if for any $\delta
>0$, $(n\phi(h))^{-\delta/2} \rightarrow 0$,  we have

$$\sqrt{ n \phi(h)} \left( \nabla_u \widetilde G_n^x(\mu) - \mathbb{E}\left[\nabla_u \widetilde G_n^x(\mu)\right]\right)
\stackrel{\mathcal{D}}{\longrightarrow} \mathcal{N}_d(0, \
\tilde{\Sigma}^x(\mu)).$$

where $\tilde{\Sigma}^x(\mu)$ is the limiting covariance matrix of  $\nabla_u \widetilde G_n^x(\mu) -
\mathbb{E}\left[\nabla_u \widetilde G_n^x(\mu)\right].$
%
\end{lemma}

\begin{proof} of Lemma \ref{l1}. Let's denote by

$$A_{i} = \frac{\sqrt{\phi(h)}}{\mathbb{E}(\Delta_1(x))} \times \mathcal{U}(Y_i-\mu)\; \Delta_i(x) $$
Then
$$\sqrt{n \phi(h)} \left( \nabla_u \widetilde G_n^x(\mu) - \mathbb{E}\left[\nabla_u \widetilde G_n^x(\mu)\right]\right)=
\frac{1}{\sqrt{n}} \sum_{i=1}^n (A_{i} - \mathbb{E}(A_{i}))
:=\frac{1}{\sqrt{n}} \sum_{i=1}^n \widetilde A_{i}. $$ From the
Cramer-Wold device, Lemma \ref{l1} can be proved by finding the
limit distribution of the real variables sequence
$\frac{1}{\sqrt{n}} \sum_{i=1}^n \ell^t\;\widetilde A_{i}$,
 for all $\ell\in\ER^d$ satisfying  $\|l\|\neq0$.

Because the random variables $\ell^t \widetilde A_{1}, \dots,
\ell^t \widetilde A_{n} $ are i.i.d. with zero mean and asymptotic
variance
\begin{eqnarray*}
\sigma^2(x) &=&
\lim_{n\rightarrow\infty}Var\left({\frac{1}{\sqrt{n}} \sum_{i=1}^n
\ell^t\;\widetilde A_{i}}\right).
\end{eqnarray*}
The result may be obtained by applying the Liapounov Central
Theorem Limit. For this propose, we have to prove the following
Lindeberg condition:
 $$\forall \delta > 0 \quad \left[n\;\ell^t\widetilde{\Sigma}^x(\mu)\ell\right]^{-(2+\delta)/2} \sum_{i=1}^n \mathbb{E}|\ell^t \widetilde A_i|^{2+\delta} \longrightarrow 0
 \quad \mbox{as} \quad n \longrightarrow \infty.$$
  It is easy to see that:
 \begin{eqnarray*}
 \left[n\;\ell^t\widetilde{\Sigma}^x(\mu)\ell\right]^{-(2+\delta)/2} \sum_{i=1}^n \mathbb{E}|\ell^t \widetilde A_i|^{2+\delta}
 &=& n^{-\delta/2} \left( \ell^t\Sigma^x(\mu)\ell\right)^{-(2+\delta)/2} \mathbb{E} |\ell^t \widetilde
 A_1|^{2+\delta}.
 \end{eqnarray*}
 Moreover,  using $C_r$ and Jensen inequalities, we obtain
 \begin{eqnarray*}
 \mathbb{E} |\ell^t \widetilde A_1|^{2+\delta} &\leq& c \frac{(\phi(h))^{(2+\delta)/2}}{(\mathbb{E}\Delta_1(x))^{2+\delta}}
 \mathbb{E}\left|\ell^t \left(\mathcal{U}(Y_1-\mu) \right)^{2+\delta} \times \Delta_1^{2+\delta}(x)\right|\\
 \end{eqnarray*}
 \begin{eqnarray*}
 &\leq& c \frac{(\phi(h))^{(2+\delta)/2}}{(\mathbb{E}\Delta_1(x))^{2+\delta}}
 \mathbb{E}\left\{\Delta_1^{2+\delta}(x) \;\underbrace{\mathbb{E}\left[\left|\ell^t\mathcal{U}(Y_1-\mu)
 \right|^{2+\delta}\;|\; X \right]}_{= W_{2+\delta}^X(\mu)} \right\}\\
 &\leq& c \frac{(\phi(h))^{(2+\delta)/2}}{(\mathbb{E}\Delta_1(x))^{2+\delta}}
  \left[\mathbb{E}(\Delta_1(x))^{2+\delta} \sup_{x'\in B(x,h)}
  |W_{2+\delta}^{x'}(\mu) - W_{2+\delta}^x(\mu)| +
  W_{2+\delta}^x(\mu) \mathbb{E}(\Delta_1(x))^{2+\delta} \right].\\
 \end{eqnarray*}
 It follows then, by  hypothesis { (H4)(ii)}  and Lemma
 \ref{lemma1}, that
 \begin{eqnarray*}
 \mathbb{E} |\ell^t \widetilde A_1|^{2+\delta}    &\leq&
 c \frac{(\phi(h))^{(2+\delta)/2}}{(\mathbb{E}\Delta_1(x))^{2+\delta}}
W_{2+\delta}^x(\mu) \mathbb{E}(\Delta_1(x))^{2+\delta} \\
 &\leq& c' \frac{\left(\phi(h) \right)^{(2+\delta)/2}}{\left(\phi(h) \right)^{(2+\delta)} [M_1^{2+\delta} (g(x))^{2+\delta}+o(1)]}
\left[\phi(h) (M_{(2+\delta)/2}g(x) +o(1)) \right]\\
&=&\mathcal{O}\left((\phi(h)^{-\delta/2}) \right).
\end{eqnarray*}

Finally, since $(\ell^t\widetilde{\Sigma}^x(\mu)\ell)^{-(2+\delta)/2}$ is
finite, it comes that

$$\left[n\;\ell^t\widetilde{\Sigma}^x(\mu)\ell\right]^{-(2+\delta)/2} \sum_{i=1}^n \mathbb{E}|\ell^t \widetilde A_i|^{2+\delta}  =
\mathcal{O}\left((n\phi(h)^{-\delta/2}) \right)=o(1),$$

because  $n\phi(h) \rightarrow \infty$ as $n\rightarrow\infty$. This
implies  the Lindeberg condition, which  completes the proof of
the Lemma.
\end{proof}

The following Lemma  gives the analytic expression of the matrix
$\Sigma^x(\mu)$.

\begin{lemma}\label{sigma2} Under conditions (H1)-(H2) and (H4)(ii), we have

$$\sigma^2(x) = \lim_{n\rightarrow\infty}Var\left({\frac{1}{\sqrt{n}} \sum_{i=1}^n \ell^t\;\widetilde A_{i}}\right) =
\frac{M_2}{M_1^2 g(x)} \; \ell^t\Sigma^x(\mu)\ell.$$

\end{lemma}


\begin{proof} of Lemma \ref{sigma2}. Since the random variables $(\ell^t\widetilde
A_i)_{i=1,\dots,n}$ are i.i.d. with mean zero, it follows that
$$\sigma^2(x) = \lim_{n\rightarrow\infty} Var\left({\frac{1}{\sqrt{n}} \sum_{i=1}^n \ell^t\;\widetilde A_{i}}\right)  =
 \lim_{n\rightarrow\infty} Var(\ell^t \widetilde A_1)= \lim_{n\rightarrow\infty} \mathbb{E}\left((\ell^t A_1)^2 \right).$$
 On the other hand, making use of the properties of conditional expectation one may write
\begin{eqnarray*}
\mathbb{E}\left[\left(\ell^t A_1 \right)^2 \right] &=&
\frac{\phi(h)}{(\mathbb{E}\Delta_1)^2} \;
\mathbb{E}\left[\Delta_1 \ell^t \mathcal{U}(Y_1-\mu)\right]^2
=\frac{\phi(h)}{(\mathbb{E}\Delta_1)^2} \;\mathbb{E}\left[\Delta^2_1   W_2^{X_1}(\mu)  \right]\\
\end{eqnarray*}
Making use of the condition (H4)(ii) and the fact that the functions
$W_2^x (\cdot)$  is bounded, we obtain
\begin{eqnarray*}
\mathbb{E}\left\{\Delta^2_1 W_2^{X_1}(\mu) \right\}  &=&
\mathbb{E}\left(\Delta^2_1\right)\left[ W_2^x(\mu) +\mathcal{O}\left(
\sup_{u\in
\mathbb{R}^d}|W^u_2(\mu)- W^x_2(\mu)| \right)\right] \\
&=& W_2^x(\mu)\mathbb{E}\left(\Delta^2_1\right)
 +o\left( \mathbb{E}\left(\Delta^2_1\right)\right).
\end{eqnarray*}

Using Lemma \ref{lemma1}, one may see that
$$\frac{\phi(h)}{(\mathbb{E}\Delta_1)^2} \mathbb{E}(\Delta_1^2)=\frac{M_2}{M_1^2 g(x)}+o(1).$$

Therefore,

$$\sigma^2(x)=\frac{M_2}{M_1^2 g(x)}W_2^x(\mu)+o(1).$$
\end{proof}


\begin{proof} of Proposition \ref{B}. For each $x \in \mathcal{F}$,
  since $(X_i,Y_i)_{i=1,\dots,n}$ are i.i.d., we
have
$$\widetilde{\mathcal{B}}_n(x) = \mathbb{E}\left[  \nabla_u \widetilde G_n^x(\mu)\right] =
\frac{\mathbb{E}\left[\mathcal{U}(Y_1-\mu) \Delta_1(x) \right]}{\mathbb{E}(\Delta_1(x))}$$
By conditioning with respect to real variable $d(x,X_1)$ and using condition (H5), we have
$$
\widetilde{\mathcal{B}}_n(x) =
 \frac{\mathbb{E}\left[ K\left(\frac{d(x,X_1)}{h}\right) \psi(d(x,X_1))\right]}{\mathbb{E}\left(K\left( \frac{d(x,X_1)}{h}\right) \right)}.
$$
Integration with respect to the distribution of the real variable
$d(x,X_1)$  shows that

$$A_1:=\mathbb{E}\left[ K\left(\frac{d(x,X_1)}{h}\right) \psi(d(x,X_1))\right] = \int_0^1 K(t) \psi(th) dF(th),$$
where $F$ is the cumulative distribution function of the real random variable $d(x,X)$. On the other hand, Taylor series expansion of the function $\psi$
up to the order one in the neighborhood of $t=0$ gives $\psi(th) =
th \nabla\psi(0)+ o_d(h).$ Let us denote by $o_d(1)$ (resp.
$\mathcal{O}_d(1)$) a d-dimensional vector where each component equal to
$o(1)$ (resp. $\mathcal{O}(1)$).

\noindent Therefore, we have
\begin{eqnarray*}
A_1&=& h\nabla\psi(0)\int_0^1 tK(t) dF_x(th) + o_d(h)\int_0^1 K(t) dF(th)\\
&=& h \nabla\psi(0) \left[K(1) F(h) - \int_0^1 (sK(s))' F(sh) ds \right] + o_d(h) \left[K(1) F(h) - \int_0^1 K'(s) F(sh) ds \right].
\end{eqnarray*}

\noindent Using hypothesis $(H2)(i)-(ii)$ we get
\begin{eqnarray*}
A_1 &=& h \nabla \psi (0) K(1) \left(\phi(h) g(x) + o(\phi(h)) \right) - h \nabla \psi(0) \int_0^1 (sK(s))' \left(\phi(sh) g(x) + o(\phi(hs)) \right) ds\\
& & + o(h) K(1) \left(\phi(h) g(x) + o(\phi(h)) \right) - o_d(h) \int_0^1 K'(s) \left(\phi(sh) g(x) + o(\phi(sh)) \right) ds\\
&=& h\phi(h) g(x) \nabla\psi(0) \left[K(1) - \int_0^1 (sK(s))' \left(\tau_0(s) + o(1) \right) ds \right] + h\phi(h) K(1) o_d(1)\\
& & - o_d(h\phi(h))  \int_0^1 K'(s) (\tau_0(s) g(x) + o(1)) ds\\
&=&h\phi(h) g(x) \nabla\psi(0) \left[K(1) - \int_0^1 (sK(s))'  \tau_0(s) ds \right] + O_d^{a.s.}(h\phi(h))
\end{eqnarray*}
\noindent Thus, making use of the Lemma \ref{lemma1}, we obtain
\begin{eqnarray*}
\widetilde{\mathcal{B}}_n(x) = \frac{h\nabla\psi(0)}{M_1} \left[K(1) - \int_0^1 (sK(s))' \tau_0(s) ds + o_{a.s.}(1) \right]
\end{eqnarray*}

\end{proof}

\begin{proof} of Theorem \ref{CLT}

Part (i) follows from Proposition \ref{variance}, decomposition
(\ref{deco}),  Proposition \ref{l1} and Lemma \ref{sigma2}.

Part (ii) follows from Proposition   \ref{B} combined with the
condition $\sqrt{n\phi(h)} h\longrightarrow 0$ as $n\rightarrow
\infty$.

\end{proof}
\begin{proof} of Corollary \ref{TCLP}. Let us denote by $$ T^{x}(\mu) = \left[\Sigma^x(\mu) \right]^{-1/2} H^x(\mu), \quad T_n^{x}(\mu_n) =
\left[\Sigma_n^x(\mu_n) \right]^{-1/2} H_n^x(\mu_n)$$
and $$V_n^{x}(\mu_n)=\frac{M_{1,n}}{\sqrt{M_{2,n}}}\sqrt{nF_{x,n}(h)}\;\;
T_n^{x}(\mu_n)  \left( \mu_{n}-\mu\right).$$

Write
\begin{eqnarray}\label{TCLP.1}
V_n^{x}(\mu_n) &= &\frac{M_{1,n}\sqrt{M_2}}{M_1\sqrt{M_{2,n}}}\;
\sqrt{nF_{x,n}(h) \left(n\phi(h)g(x)\right)^{-1}}\; T_n^{x}(\mu_n)\; \left[T^{x}(\mu)\right]^{-1}
\times \frac{M_1}{\sqrt{M_2}}\sqrt{n\phi(h)g(x)}\;T^{x}(\mu) \left( \mu_{n}-\mu\right)
\nonumber
\\
&:=&V_{n,1}^x\times V_{n,2}^x.
 \end{eqnarray}
Making use of Theorem \ref{CLT} part (ii), the term $V_{n,2}^x$ converges in
distribution to ${\cal N}( 0, I_d)$.

Now to get the result of the corollary it suffices to show that
the first term $V_{n,1}^x$ converges to 1 in probability.
Following the same arguments as in \cite{laib2010} combined with (H1),(H2), one gets \\

$\displaystyle\frac{M_{1,n}\sqrt{M_2}}{M_1\sqrt{M_{2,n}}}\;
\sqrt{nF_{x,n}(h) \left(n\phi(h)g(x)\right)^{-1}} \stackrel{\mathbb{P}}{\longrightarrow} 1$, $M_{1,n} \stackrel{\mathbb{P}}{\longrightarrow} M_{1}$ and $M_{2,n} \stackrel{\mathbb{P}}{\longrightarrow} M_{2}$, as $n \rightarrow\infty$.\\

Now, we have to establish the consistency of $T_n^{x}(\mu_n)$. To
do that, we will  study separately the consistency of each term of
$T_n^{x}(\mu_n)$. Let us start by $H_{n}^x(\mu_{n})$. For this,
write

\begin{eqnarray*}
H_n^x(\mu_{n})-H^x(\mu) &=& \frac{\widetilde{H}_n^x(\mu_{n})}{G_{n,1}^x} - H^x(\mu) \nonumber\\
&=& \frac{\widetilde{H}_n^x(\mu_{n}) - H(\mu)}{G_{n,1}^x} + \frac{H^x(\mu) (1-G_{n,1}^x)}{G_{n,1}^x}.
\end{eqnarray*}
According to Theorem \ref{thm2}, Proposition \ref{variance}, Lemma 5.2 and the fact that the matrix $H^x(\mu)$ is bounded, we can conclude that $H_n^x(\mu_{n})$ converges, in probability, to $H^x(\mu)$.
%

The second term $\Sigma_n^x(\mu_n)$, can be treated  similarly.
Finally, this leads to the convergence in probability of
$T_n^{x}(\mu_n)$ to $T^{x}(\mu)$.
\end{proof}






\end{document}